\documentclass[10pt,a4paper]{article}
\usepackage{authblk}

\usepackage{amsmath}
\usepackage{amsfonts}
\usepackage{amssymb}
\usepackage{amsthm}
\usepackage{ulem}
\usepackage{xcolor}
\usepackage{verbatim}
\usepackage{dsfont}
\numberwithin{equation}{section}

\newtheorem{thm}{Theorem}[section]
\newtheorem{pro}[thm]{Proposition}
\newtheorem{lem}[thm]{Lemma}
\newtheorem{cor}[thm]{Corollary}

\theoremstyle{Definition}
\newtheorem{dfn}[thm]{Definition}

\newtheorem{rem}[thm]{Remark}
\theoremstyle{plain}

\begin{document}

\title{Mixing properties and the chromatic number \\  of Ramanujan complexes}
\author{Shai Evra\thanks{E-mail address: \texttt{shai.evra@mail.huji.ac.il}} , 
Konstantin Golubev\thanks{E-mail address: \texttt{kost.golubev@mail.huji.ac.il}} ,
Alexander Lubotzky\thanks{E-mail address: \texttt{alex.lubotzky@mail.huji.ac.il}} }
\affil{Institute of Mathematics \\ Hebrew University \\ Jerusalem 91904  ISRAEL}
\maketitle

\begin{center}
\textit{Dedicated to Nati Linial on his 60th birthday.}
\end{center}

\begin{abstract}
Ramanujan complexes are high dimensional simplical complexes generalizing Ramanujan graphs.  
A result of Oh on quantitative property $(T)$ for Lie groups over local fields is used to deduce a Mixing Lemma for such complexes. 
As an application we prove that non-partite Ramanujan complexes have 'high girth' and high chromatic number, generalizing a well known result about Ramanujan graphs. 
\end{abstract}

%%%%%%%%%%%%%%%%%%%%%%%%%%%%%%%%%%%%%%%%%%%%%%%%%%%%%%%%%%%%%%%%%%%%%%%%%%%%%%%%%%%%%%%%%%%%%%%%%%%%%%%%%%%%%%%%%%%%%%%%%%%%%%%%%%%%%%%%%%%%%%%%%%%%%%%%%%%%%%%%%%%%%%%%%%%%%%%%%%
%%%%%%%%%%%%%%%%%%%%%%%%%%%%%%%%%%%%%%%%%%%%%%%%%%%%%%%%%%%%%%%%%%%%%%%%%%%%%%%%%%%%%%%%%%%%%%%%%%%%%%%%%%%%%%%%%%%%%%%%%%%%%%%%%%%%%%%%%%%%%%%%%%%%%%%%%%%%%%%%%%%%%%%%%%%%%%%%%%

\section{Introduction}

In 1959, Erd\H{o}s \cite{E} used random methods to show that there are graphs with arbitrary large girth and arbitrary large chromatic number. 
In a way this is a surprising fact, since large girth means that such a graph looks locally like a tree, and so locally its chromatic number is two, 
while globally it requires a large number of colors. A constructive proof was given by Lov\'{a}sz \cite{Lo} in 1968, 
and explicit examples (with quantitative estimates) in 1988 by Lubotzky, Phillips, Sarnak \cite{LPS} using Ramanujan graphs. 
This is still by no mean an easy result even by nowadays standards.  

The goal of this paper is to extend the above from Ramanujan graphs to high-dimensional Ramanujan complexes, as defined and constructed in \cite{LSV1} and \cite{LSV2}.
One is facing the immediate question what we mean by "girth" and "chromatic number" for simplicial complexes?

The girth $g(X)$ of a graph $X$ is equal to twice its injectivity radius $r(X)$ (more precisely $r(X) = \lfloor \frac{g(X) - 1}{2} \rfloor$). 
The injectivity radius of $X$ is the maximal $r \in \mathbb{N}$ such that if $\pi : \tilde{X} \rightarrow X$ is the universal cover map, then for every $y \in \tilde{X}$,
$\pi$ is one-to-one on the ball of radius $r$ around $y$. This definition is easily extended to finite simplicial complexes
and, in particular, to the Ramanujan complexes, whose universal covers are the Bruhat-Tits buildings. 
The injectivity radius is defined similarly with respect to the graph metric on the $1$-skeletons of $\tilde{X}$ and $X$.
This notion has been studied in \cite{LuMe} where it was shown that there exist Ramanujan complexes of "large girth" in this sense.
See Proposition 3.3 there and Corollary \ref{inj-rad} below.

Before moving to the chromatic number, let us make the following warning:
Every simplicial complex $X$ can be considered as a hypergraph $H$, when we take the maximal simplices (facets) of $X$ to be the edges of $H$. 
Moreover if $X$ is pure, i.e. all its facets are of the same dimension, say $d-1$,
then $H$ is a $d$-uniform hypergraph, i.e. all of its edges are of size $d$.
The commonly used notion of girth in the theory of hypergraphs is different then the one we are using here; it refers to the length of a minimal sequence of the form 
$x_1,E_1,x_2,E_2,\ldots,x_g,E_g,x_{g+1}$ where all $x_1,\ldots,x_g$ are different vertices, $x_{g+1}=x_1$, $E_1,\ldots,E_g$ are edges
and for any $i=1,\ldots,g$, $\{ x_i,x_{i+1} \} \subset E_i$.
This notion is not suitable for the Ramanujan complexes or any clique complex: 
any two facets with a common $1$-codimension wall give girth $2$ in this definition (so, even the building has girth $2$).
Anyway, the theory of hypergraphs of high girth and high chromatic number has been developed quite intensively.
See \cite{N} for a nice survey. The reader is referred also to (\cite{LuMe}, \cite{G1} and \cite{G2}) for related notions of girth for simplicial complexes,
based on local acyclicity.

The notion of chromatic number for simplicial complexes we will use is the same as the one commonly used for hypergraphs.
Let $X$ be a $(d-1)$-dimensional simplicial complex with a set of vertices $V$.
\begin{dfn} \label{dfn-cor}
The chromatic number of $X$, denoted $\chi(X)$, is the minimal number of colors needed to color the vertices of $X$, so that no facet (i.e. maximal face) is monochromatic.
\end{dfn}
Clearly $\chi(X)$ is bounded above by the chromatic number of the graph $X^{(1)}$(= the 1-skeleton of $X$).

Let us now recall what are Ramanujan complexes and how they are constructed:
Let $F$ be the local non-archimedean field $\mathbb{F}_{q}((t))$, i.e. the field of Laurent power series over $\mathbb{F}_q$,
where $\mathbb{F}_{q}$ is the finite field of order $q$. Let $\mathcal{B}=\mathcal{B}_{d}(F)$
be the Bruhat-Tits building associated with $PGL_{d}(F)$. It is an infinite $(d-1)$-dimensional countable simplicial complex, whose
vertices come naturally with types in $\mathbb{Z} / d\mathbb{Z}$, denoted $\tau: \mathcal{B}(0) \rightarrow \mathbb{Z} / d\mathbb{Z}$ (see \cite{LSV1}
and the references therein) in such a way that in every $(d-1)$-face all vertices are of different types. In particular, its chromatic
number is at most $d$. (Even its 1-skeleton has chromatic number $d$.) In fact, its chromatic number is 2, since we can divide the
set of $d$ types $\mathbb{Z} / d\mathbb{Z}$ into two non-empty disjoint sets and then using only two colors, we get that no $(d-1)$-cell
is monochromatic. If $\Gamma$ is a cocompact lattice in $G=PGL_{d}(F)$, with $dist(\Gamma) := \min_{1 \ne \gamma \in \Gamma, x \in \mathcal{B}}dist(\gamma.x,x) \geq 2$, 
then $\Gamma\backslash\mathcal{B}$ is a finite simplicial complex. 
If $\Gamma$ preserves the types of the vertices of $\mathcal{B}$, $\Gamma\backslash\mathcal{B}$ is also $d$-colorable (and even 2-colorable). 
The injectivity radius of $\Gamma \backslash \mathcal{B}$ is $\lfloor \frac{dist(\Gamma)-1}{2} \rfloor$.

We will use the remarkable lattice $\Lambda$ constructed by Cartwright and Steger \cite{CS}, which acts transitively on the vertices of $\mathcal{B}$, and, 
in particular, does not preserve the types of the vertices of $\mathcal{B}$. In this $\Lambda$ we will choose suitable congruence
subgroups $\Lambda(f)$ for some $f \in \mathbb{F}_{q}[\frac{1}{t}]$, and show:

\begin{thm} \label{chromatic-bound} 
For every integer $d \geq 3$ and odd prime power $q$, there exists a sequence of finite $(d-1)$-dimensional simplicial complexes $(X_n)_{n \in \mathbb{N}}$ 
with $|X_n| \rightarrow \infty$, covered by $\mathcal{B}_{d}(\mathbb{F}_{q}((t)))$, with injectivity radius
\begin{equation}\nonumber
r(X_n) \geq \frac{\log_q|X_n|}{2(d-1)(d^2-1)} - \frac{1}{2}
\end{equation}
(so, the chromatic number of every ball of radius $\frac{\log_q|X_n|}{2(d-1)(d^2-1)} - \frac{1}{2}$ is two), while 
\begin{equation}\nonumber
 \chi(X_n)\geq \frac{1}{2} \cdot q^{\frac{1}{2d}} 
\end{equation}
and so, $\chi(X_n)\to\infty$ when $q\to\infty$. In particular, by letting $q \rightarrow \infty$, this gives for every $d \geq 3$, 
$(d-1)$-dimensional simplicial complexes of arbitrary large "girth" (twice the injectivity radius) and arbitrary large chromatic number.
\end{thm}

Note, that in order to have arbitrarily large chromatic number, $q$ must go to infinity, otherwise the chromatic number of $X_n$, even as graphs, 
would be bounded since the degree would be bounded. 

Moreover, for our complexes 
$$diam(X_n) \leq \frac{4 \log_q|X_n|}{d^2} \leq 8d \cdot r(X_n),$$ 
for $q \gg d$ (see Remark \ref{diameter}). 
In particular, up to radius $\frac{diam(X_n)}{8d}$, the chromatic number of a ball in $X_n$ is $2$, and only for bigger balls it grows, 
eventually to infinity.

As mentioned before, the fact that the quotients by congruence subgroups give large injectivity radius (and no small non-trivial homology cycles) was shown by Lubotzky
and Meshulam in \cite{LuMe} (see \cite[Section~4.1]{GuLu} for a "general principle" of this kind). So the main novelty of the current
paper is giving a lower bound on the chromatic number for some carefully chosen congruence subgroups (see \S 5.3 below). To this end we will prove the following result
which is of independent interest:

\begin{thm}[Colorful Mixing Lemma] \label{colormix} 
Let $F$ be a non-archimedean local field with finite residue field $\mathbb{F}_{q}$, $q$ odd, $d\geq 3$, and $\mathcal{B}=\mathcal{B}_d(F)$, 
the Bruhat-Tits building associated with $PGL_d(F)$.  Let $\Gamma \leq PGL_{d}(F)$ be a cocompact lattice preserving the type (coloring) function of $\mathcal{B}(0)$ with injectivity radius $\geq 2$, so $X=\Gamma\backslash\mathcal{B}$ is a simplicial complex with a type function $\tau : X(0) \rightarrow \mathbb{Z}/d\mathbb{Z}$. 
For each type $i \in \mathbb{Z}/d\mathbb{Z}=\{1,2,\ldots,d\}$, let $V_i \subset X(0)$ be the set of vertices of type $i$, i.e. $V_i=\tau^{-1}(i)$.

Then for any choice of subsets $W_i \subseteq V_i$ we have: 
\begin{equation}\nonumber
\left|\frac{|E(W_1,\dots,W_d)|}{|X(d-1)|}-\prod_{i=1}^d\frac{|W_i|}{|V_i|}\right|\leq\frac{2d}{q^{\frac{1}{2}}}
\end{equation}
where $E(W_1,\dots, W_d)$ is the set of all $(d-1)$-dimensional cells with exactly one vertex in each $W_i, i=1,\dots, d$.
\end{thm}

So, the lemma ensures that when $q \gg 0$, the number $|E(W_1,\dots,W_d)|$ of facets with one vertex from each $W_i$ 
is approximately what one should expect by random considerations.

This mixing lemma will be deduced from a more general one (see Corollary \ref{disc} below) using a result of Hee Oh \cite{Oh} 
which gives a quantitative estimate for Kazhdan property (T) of $PGL_{d}(F)$. In this argument we follow a related use of Oh's work in \cite{FGLNP}.

It is interesting to observe that the above mixing lemma is for quotients of $\mathcal{B} = \mathcal{B}_{d}(F)$ on which the $d$-coloring by the $d$ types is preserved, 
but eventually our main theorem is about $X_{f}=\Lambda(f) \backslash \mathcal{B}$ which are not $d$-colorable (in fact, our main goal is to show that they need many more colors!) We will acquire this by applying the colorful mixing lemma to the natural $d$-colorable $d$-sheeted cover of $X_{f}$ (see Section 5 for details). 

The paper is organized as follows. 
After a few preliminaries in Section 2,  
we show in Section 3 how the discrepancy of a colorful simplicial complex can be estimated using the eigenvalues of some naturally associated bipartite graphs. 
In Section 4, we will use Oh's theorem and apply it to the colorable quotients of the Bruhat-Tits building of $PGL_{d}(F)$, to estimate these eigenvalues. 
In Section 5, we will follow carefully \cite{LSV2} to choose the suitable congruence subgroups $\Lambda(f)$ of $\Lambda$ --- the Cartwright-Steger lattice. 
We will use the congruence subgroups $\Lambda(f)$ for which $\Lambda(f) \backslash \mathcal{B}$ is a non-partite complex, see there. 
In Section 6, we collect all the information together and prove Theorem \ref{chromatic-bound}.

This paper is dedicated to Nati Linial who has pioneered the study of high dimensional expanders and many other things.

\textit{Acknowledgement.} The authors are grateful to the ERC, ISF and NSF for partial support.
This work is part of the Ph.D. theses of the first two authors at the Hebrew University of Jerusalem, Israel.
The authors are grateful to Nati Linial for valuable discussions.

%%%%%%%%%%%%%%%%%%%%%%%%%%%%%%%%%%%%%%%%%%%%%%%%%%%%%%%%%%%%%%%%%%%%%%%%%%%%%%%%%%%%%%%%%%%%%%%%%%%%%%%%%%%%%%%%%%%%%%%%%%%%%%%%%%%%%%%%%%%%%%%%%%%%%%%%%%%%%%%%%%%%%%%%%%%%%%%%%%
%%%%%%%%%%%%%%%%%%%%%%%%%%%%%%%%%%%%%%%%%%%%%%%%%%%%%%%%%%%%%%%%%%%%%%%%%%%%%%%%%%%%%%%%%%%%%%%%%%%%%%%%%%%%%%%%%%%%%%%%%%%%%%%%%%%%%%%%%%%%%%%%%%%%%%%%%%%%%%%%%%%%%%%%%%%%%%%%%%

\section{Notations and conventions}\label{sec:notations}
Throughout this paper $H$ is a finite $d$-uniform hypergraph, i.e. $H=(V,E)$ and $E \subset {V\choose d}$. We say that $H$ has a $d$-type function $\tau:V=V(H)\to \mathbb{Z}/d\mathbb{Z}$ if each edge contains vertices of all $d$ types, i.e. $\tau$ is one-to-one when restricted to each edge $e\in E$. We call such hypergraph $d$-partite and we also write it as $H=(V_0,\dots,V_{d-1},E)$, where $V_i=\tau^{-1}(i),i\in\mathbb{Z}/d\mathbb{Z}$, and so $E$ can be considered as a subset of $\prod_{i=0}^{d-1}V_i$. 
A $2$-partite hypergraph is what is usually called a bipartite graph.
Sometimes it is more convenient to think of $\mathbb{Z}/d\mathbb{Z}$ as $\{1,\ldots,d\}$.

Recall that a simplicial complex $X=(V,E)$ is a family $E$ of finite subsets (called faces or simplicies) of the set of vertices $V$ closed under inclusion, 
i.e. if $F_1\in E$ and $F_2\subseteq F_1$ then $F_2\in E$.

For $F\in E$, denote $\dim (F) = |F|-1$ and $X(i)$ the set of simplices of dimension $i$. We say that $\dim (X) = d$ if $X(d) \neq \emptyset$ while $X(d+1)=\emptyset$. 
It is called a pure complex of dimension $d$, if every maximal face in $E$ is of dimension $d$. Given $X=(V,E)$, we denote by $X^{(i)}$ the $i$-skeleton of $X$, this is the subcomplex of $X$ of all the faces $F$ in $E$ with $\dim (F)\leq i$.

Given a pure simplicial complex $X=(V,E)$ of dimension $(d-1)$, one can associate with it the $d$-uniform hypergraph $H=\widetilde{H}(X)=(V,X(d-1))$. 
Conversely if $H=(V,E)$ is a $d$-uniform hypergraph then by taking $\tilde{E}=\{F \subseteq V\mid \exists\, e\in E \text{ with } F\subseteq e\}$ 
we get a pure simplicial complex $X=\widetilde{X}(H)=(V,\tilde{E})$ of dimension $d-1$. Clearly, $\widetilde{X}(\widetilde{H}(X))=X$ and $\widetilde{H}(\widetilde{X}(H))=H$. 
Moreover if $\tau$ is a type function on $H$, it defines a type function on $X$ such that when restricted to every maximal face (facet) it is one-to-one.
Such complexes are called balanced.

The theories of pure simplicial complexes and uniform hypergraphs are therefore completely equivalent. In this paper, we will use these languages alternately.

%%%%%%%%%%%%%%%%%%%%%%%%%%%%%%%%%%%%%%%%%%%%%%%%%%%%%%%%%%%%%%%%%%%%%%%%%%%%%%%%%%%%%%%%%%%%%%%%%%%%%%%%%%%%%%%%%%%%%%%%%%%%%%%%%%%%%%%%%%%%%%%%%%%%%%%%%%%%%%%%%%%%%%%%%%%%%%%%%%
%%%%%%%%%%%%%%%%%%%%%%%%%%%%%%%%%%%%%%%%%%%%%%%%%%%%%%%%%%%%%%%%%%%%%%%%%%%%%%%%%%%%%%%%%%%%%%%%%%%%%%%%%%%%%%%%%%%%%%%%%%%%%%%%%%%%%%%%%%%%%%%%%%%%%%%%%%%%%%%%%%%%%%%%%%%%%%%%%%

\section{Discrepancy}
For a $d$-partite hypergraph $H = (V_1,\ldots,V_d,E)$, and a collection of subsets $W_i \subseteq V_i, i=1,\dots,d$, denote $E(W_1,\ldots,W_d) = E\cap \prod_{i=1}^d W_i$, the edges in $E$ with vertices in $W_1,\ldots,W_d$. We define the \textit{discrepancy} of $W_1,\ldots,W_d$ in $H$ to be 
\begin{equation}
disc_H(W_1,\ldots,W_d) = \left| \frac{|E(W_1,\ldots,W_d)|}{|E|} - \prod_{i=1}^d\frac{|W_i|}{|V_i|} \right| 
\end{equation}
In other words, the discrepancy measures the difference between the actual portion of edges between $W_1,\ldots,W_d$ and the expected portion if the hyperedges would have been chosen randomly uniformly.

For a biregular bipartite graph, the expander mixing lemma provides an upper bound on the discrepancy in the terms of the the second largest eigenvalue of the graph. In this section our aim is to give a similar bound for $d$-partite hypergraphs.

%%%%%%%%%%%%%%%%%%%%%%%%%%%%%%%%%%%%%%%%%%%%%%%%%%%%%%%%%%%%%%%%%%%%%%%%%%%%%%%%%%%%%%%%%%%%%%%%%%%%%%%%%%%%%%%%%%%%%%%%%%%%%%%%%%%%%%%%%%%%%%%%%%%%%%%%%%%%%%%%%%%%%%%%%%%%%%%%%%
\subsection{Discrepancy of bipartite graphs}
Let $G=(V_1,V_2,E)$ be a finite connected bipartite $(k_1,k_2)$-biregular graph on $n$ vertices, i.e. each vertex in $V_1$ has exactly $k_1$ neighbors, 
all of them in $V_2$, and each vertex in $V_2$ has $k_2$ neighbors, all of them in $V_1$, $|V_1|+|V_2|=n$, and so $k_1|V_1| = k_2|V_2| = |E|$.

Recall that the adjacency operator $A=A(G)$ of the graph $G$ is the following operator on the space of complex valued functions on the vertices
\begin{equation}
(Af)(v)=\sum_{u\sim v} f(u), 
\end{equation}
where $u\sim v$ stands for $(u,v)\in E$.

The following lemmas are probably known, but for lack of a reference we give short proofs. 

\begin{lem} \label{3.1}
Let $\lambda_n \leq \ldots \leq \lambda_2 \leq \lambda_1$ be the eigenvalues of the adjacency operator $A$ of $G$. Then
\begin{enumerate}
\item \label{symmetry-of-the-spectrum}The spectrum is symmetric, i.e. $\lambda_{n-i+1} = -\lambda_i$ for all $i$.
\item \label{the-largest-ev} The largest (resp. smallest) eigenvalue is $\lambda_1 = \sqrt{k_1k_2}$ (resp., $\lambda_n = -\sqrt{k_1 k_2}$), 
whose corresponding eigenfunction is $\sqrt{k_1}\mathds{1}_{V_1} + \sqrt{k_2}\mathds{1}_{V_2}$ (resp. $\sqrt{k_1}\mathds{1}_{V_1} - \sqrt{k_2}\mathds{1}_{V_2}$).
\end{enumerate} 
\end{lem}

\begin{proof}
Let $f$ be an eigenfunction of $A$ with an eigenvalue $\lambda$, i.e. $Af=\lambda\cdot f$. Then it is easy to see that the following function 
\begin{equation}
  g(v)=\begin{cases}
       f(v),&\text{if }v\in V_1\\
       -f(v),&\text{if }v\in V_2      
      \end{cases}
\end{equation}
satisfies $Ag=(-\lambda)\cdot g$, which proves (1).
 
If $\lambda$ is an eigenvalue of $A$, then $\lambda^2$ is an eigenvalue of $A^2$. The operator $A^2$ expresses the 2-step walk on $G$, 
i.e. the $(v,u)$-entry in the matrix of $A^2$ equals the number of paths of length 2 in $G$ connecting the vertices $v$ and $u$. 
By the $(k_1,k_2)$-regularity condition, the sum of any row of the matrix $A^2$ is $k_1k_2$. 
Thus, $A^2$ is the adjacency matrix of a $k_1 k_2$-regular multigraph, 
i.e. a graph with loops and multiple edges, and hence its largest eigenvalue is $k_1 k_2$. 
Thus the largest eigenvalue of $A$ is $\sqrt{k_1 k_2}$, and by (1), the smallest is $-\sqrt{k_1 k_2}$.
 
By the biregularity condition,
\begin{equation}
A (\mathds{1}_{V_1}) = k_2 \mathds{1}_{V_2} \quad \text{and} \quad A (\mathds{1}_{V_2}) = k_1 \mathds{1}_{V_1}.
\end{equation}
Hence  
\begin{equation}
A(\sqrt{k_1}\mathds{1}_{V_1} \pm \sqrt{k_2}\mathds{1}_{V_2}) = \pm \sqrt{k_1 k_2}(\sqrt{k_1}\mathds{1}_{V_1} \pm \sqrt{k_2}\mathds{1}_{V_2}).
\end{equation}
\end{proof}

\begin{lem}\label{expander-mixing-lemma}[Expander mixing lemma for bipartite graphs.]
Let $G=(V_1,V_2,E)$ be a bipartite $(k_1,k_2)$-biregular finite connected graph. 
Let $\lambda = \lambda(G)$ be the second largest eigenvalue of the adjacency operator of $G$.
Then for any $S \subseteq V_1$ and $T \subseteq V_2$,
\begin{equation}\label{bipartite}
\left| |E(S,T)|-\frac{\sqrt{k_1k_2}|S||T|}{\sqrt{|V_1||V_2|}}\right| \leq \lambda(G) \sqrt{|S||T|}
\end{equation}
\end{lem}

\begin{proof}
As before, denote by $A$ the adjacency matrix of $G$, and its spectrum by $\lambda_n \leq \ldots \leq \lambda_2 \leq \lambda_1$. Note that 
\begin{equation}
 |\lambda_1 |=|\lambda_n | = \sqrt{k_1 k_2}\quad\text{and}\quad |\lambda_i|\leq \lambda(G)\quad\text{for }i=2,\dots,n-1.
\end{equation}
Let $f_1,\ldots,f_n$ be an orthonormal basis of eigenfunctions of $A$, i.e. $Af_i=\lambda_i f_i$ and $\langle f_i,f_j \rangle=\delta_{i,j}$, 
where the inner product is defined as
\begin{equation}
\langle f,g \rangle =\sum_{v\in V_1\sqcup V_2} f(v)\overline{g(v)}. 
\end{equation}

Let $\mathds{1}_{S}$ and $\mathds{1}_{T}$ be the characteristic functions of $S$ and $T$, respectively. 
Then $|E(S,T)| = \langle A \mathds{1}_{S} , \mathds{1}_T \rangle$. 
Expressing them as linear combinations of orthonormal eigenvectors of $A$ 
\begin{equation}
\mathds{1}_{S}=\sum_{i=1}^{n} s_i f_i \quad \text{and}\quad \mathds{1}_{T} = \sum_{i=1}^{n} t_i f_i,
\end{equation}
we get
\begin{equation}
|E(S,T)| = \langle A \mathds{1}_S , \mathds{1}_T \rangle =  \sum_{i=1}^{n} s_i \overline{t_i} \lambda_i  
= s_1 \overline{t_1} \lambda_1  + s_n \overline{t_n} \lambda_n  + \sum_{i=2}^{n-1} s_i \overline{t_i} \lambda_i . 
\end{equation}
By Lemma \ref{3.1}, we can assume that 
\begin{equation}
f_1 = \frac{1}{\sqrt{k_1 |V_1| + k_2 |V_2|}}(\sqrt{k_1}\mathds{1}_{V_1} + \sqrt{k_2}\mathds{1}_{V_2})
\end{equation}
and
\begin{equation}
f_n = \frac{1}{\sqrt{k_1 |V_1| + k_2 |V_2|}}(\sqrt{k_1}\mathds{1}_{V_1} - \sqrt{k_2}\mathds{1}_{V_2}), 
\end{equation}
Hence, as $k_1|V_1| = k_2|V_2|$, we get
\begin{equation}
s_1 = \langle f_1,\mathds{1}_S \rangle = \frac{|S|}{\sqrt{2|V_1|}} = s_n \quad \text{and} 
\quad \overline{t_1} = \langle f_1,\mathds{1}_T \rangle = \frac{|T|}{\sqrt{2|V_2|}} = -\overline{t_n}. 
\end{equation}
Therefore, since $\lambda_1 = \sqrt{k_1 k_2} = -\lambda_n$,
\begin{equation}
s_1\overline{t_1} \lambda_1 = s_n \overline{t_n} \lambda_n = \frac{\sqrt{k_1 k_2}|S| |T|}{2\sqrt{|V_1||V_2|}}.
\end{equation}
And so,
\begin{equation}
\begin{split}
\left||E(S,T)| - \frac{\sqrt{k_1 k_2}|S| |T|}{\sqrt{|V_1||V_2|}} \right|= \left| \sum_{i=2}^{n-1} s_i \overline{t_i} \lambda_i \right| \leq\lambda(G) \sum_{i=2}^{n-1}\left|s_i \overline{t_i} \right|\leq\\ \leq \lambda(G)\sqrt{(\sum_{i=2}^{n-1}|s_i|^2)(\sum_{i=2}^{n-1}|t_i|^2)} \leq \lambda(G) \sqrt{\|\mathds{1}_S\|^2\|\mathds{1}_T\|^2} \leq \lambda(G) \sqrt{|S||T|}.
\end{split}
\end{equation}
\end{proof}

We learned recently that Lemma \ref{expander-mixing-lemma} appears also in \cite{BAV}.

%discrepancy of bipartie biregualr graph
Recall that the normalized adjacency operator $\widetilde{A}=\widetilde{A}(G)$ of a graph $G$ is the following operator on the space of complex valued functions on the vertices
\begin{equation} \label{adj-def}
(\widetilde{A}f)(v)=\frac{1}{\deg v}\sum_{u\sim v} f(u), 
\end{equation}
where $u\sim v$ stands for $(u,v)\in E$.

For a biregular bipartite graph $G$ if $f$ is an eigenfunction of the adjacency operator $A(G)$ with an eigenvalue $\lambda$, than $\frac{1}{\sqrt{\deg v}}f(v)$ is an eigenfunction of the normalized adjacency operator $\widetilde{A}(G)$ with an eigenvalue $\frac{\lambda}{\sqrt{k_1 k_2}}$. 

In particular, the largest and smallest eigenvalues of the normalized adjacency operator are $1$ and $(-1)$, respectively. 
The second largest eigenvalue $\tilde{\lambda}$ of $\tilde{A}$ is equal to 
\begin{equation} \label{norm-ev}
\widetilde{\lambda}(G)=\frac{\lambda(G)}{\sqrt{k_1 k_2}}.
\end{equation}

\begin{dfn}
Let $G=(V_1,V_2,E)$ be a bipartite graph, $S \subseteq V_1$ and $T \subseteq V_2$. Then the discrepancy of these subsets is defined to be 
\begin{equation}
disc_G(S,T) = \left|\frac{|E(S,T)|}{|E|}-\frac{|S|}{|V_1|}\frac{|T|}{|V_2|}\right|.
\end{equation}
\end{dfn}

In these terms, the following statement is a corollary of the Expander Mixing Lemma (Lemma \ref{expander-mixing-lemma}). 
\begin{cor} \label{bipar}
Let $G=(V_1,V_2,E)$ be a bipartite $(k_1,k_2)$-biregular finite graph. Then for any $S \subseteq V_1, T \subseteq V_2$ 
\begin{equation}
disc_G(S,T) \leq \tilde{\lambda}(G) \cdot \sqrt{\frac{|S|}{|V_1|} \frac{|T|}{|V_2|}}
\end{equation}
\end{cor}

\begin{proof}
By the $(k_1,k_2)$-biregularity of $G$ 
\begin{equation}
|E|=k_1|V_1|=k_2|V_2|=\sqrt{k_1k_2|V_1||V_2|}, 
\end{equation} 
hence 
\begin{equation}
\frac{\sqrt{k_1k_2}|S||T|}{\sqrt{|V_1||V_2|}} = |E| \cdot \left(\frac{|S||T|}{|V_1||V_2|}\right).
\end{equation}
And therefore, by Lemma \ref{expander-mixing-lemma}, we get 
\begin{equation}
\left|\frac{|E(S,T)|}{|E|}-\frac{|S|}{|V_2|}\frac{|T|}{|V_1|}\right| \leq \frac{\lambda(G)}{\sqrt{k_1k_2}} \sqrt{\frac{|S||T|}{|V_1||V_2|}}.
\end{equation}
\end{proof}

%%%%%%%%%%%%%%%%%%%%%%%%%%%%%%%%%%%%%%%%%%%%%%%%%%%%%%%%%%%%%%%%%%%%%%%%%%%%%%%%%%%%%%%%%%%%%%%%%%%%%%%%%%%%%%%%%%%%%%%%%%%%%%%%%%%%%%%%%%%%%%%%%%%%%%%%%%%%%%%%%%%%%%%%%%%%%%%%%%
\subsection{Discrepancy of hypergraphs}

%discrepancy and reduction lemma 

Let $H=(V_1,\ldots,V_d,E)$ be a $d$-partite hypergraph. Our aim is to give  estimates and bounds on its discrepancy.  
We will do it by defining various associated bipartite graphs and then we will bound the discrepancies of $H$ by their discrepancies.

For $i=1,\dots,d$, denote $E_i=\{F\setminus \{v_i\}\mid F\in E,v_i\in V_i\}$, i.e. the set $E_i$ is the set of of all edges of $H$ with the vertex of type $i$ being removed.
A set $Y \in E_i$ is called a \textit{wall of cotype $i$}. Denote by $H_i=(V_1,\dots,V_{i-1},V_{i+1},\dots,V_d,E_i)$ the $(d-1)$-partite hypergraph induced from $H$ by removing the vertices of type $i$.

Denote by $B_i$ the bipartite graph with $V_i$ as one set of vertices and $E_i$ as the second. 
A vertex $v_i\in V_i$ and a wall $Y \in E_i$ are connected by an edge of $B_i$ if their union forms an edge of $H$, i.e.$\{v_i\}\cup Y \in E$. 
We will write $B_i=(V_i,E_i,E_{B_i})$. An edge of $B_i$ is a pair $(v_{i},F\setminus\{v_i\})$, where $F\in E$ is an edge of $H$. 
Following the terminology of simplicial complexes, this is the "vertices versus walls" graph. 
Note that since every edge of $H$ has exactly one vertex in $V_i$, there is a natural bijection between the edges of $B_i$ and the edges of $H$.

As before, for a collection of subsets $W_j\subseteq V_j$ for $j=1,\dots,d$, we denote by $E(W_1,\dots,W_d)$ the set of all edges of $H$ with vertices in the sets $W_1,\dots,W_d$. Analogously, for $H_i$ we will denote by $E_i(W_1,\dots, W_d)$ the subset of all edges with vertices in $W_1,\dots,W_{i-1},W_{i+1},\dots,W_d$. 
For the graph $B_i$ we will denote by $E_{B_i}(W_i,E_i(W_1,\dots, W_d))$ the set of all edges of $B_i$ with one vertex in $W_i$ and the other in $E_i(W_1,\dots,W_d)$. 
Note that the above mentioned bijection between the edges of $H$ and $B_i$ restricts to a bijection between $E(W_1,\dots,W_d)$ and $E_{B_i}(W_i,E_i(W_1,\dots, W_d))$.

The following lemma reduces the question of bounding the discrepancy of a $d$-partite hypergraph to its induced hypergraphs and bipartite graphs.
\begin{lem} \label{reduction}
Let $W_j \subseteq V_j$ for $j=1,\dots,d$. Then for $i=1,\dots,d$,
\begin{multline}
disc_H(W_1,\dots, W_d) \leq \\
disc_{B_i}(W_i,E_i(W_1,\dots,W_d)) + \frac{|W_i|}{|V_i|}disc_{H_i}(W_1,\dots,W_{i-1},W_{i+1},\dots,W_d)
\end{multline}
\end{lem}

\begin{proof}
By the definition of the bipartite graph $B_i$
\begin{equation}
\frac{|E(W_1,\dots,W_d))|}{|E|}  = \frac{|E_{B_i}(W_i,E_i(W_1,\dots,W_d))|}{|E_{B_i}|} 
\end{equation}
and hence:
\begin{multline}\nonumber
disc_{H}(W_1,\dots,W_d) = \left|\frac{|E(W_1,\dots,W_d)|}{|E|} - \prod_{j=1}^{d}\frac{|W_j|}{|V_j|} \right| = \\
= \left|\frac{|E(W_1,\dots,W_d)|}{|E|} - \frac{|W_i|}{|V_i|}\cdot \frac{|E_i(W_1,\dots,W_d)|}{|E_i|} \right.  + \\
+ \left. \frac{|W_i|}{|V_i|} \cdot \left( \frac{|E_i(W_1,\dots,W_d)|}{|E_i|} - \prod_{j=1,j\neq i}^{d}\frac{|W_j|}{|V_j|} \right) \right| \leq \\
\leq \left|\frac{|E_{B_i}(W_1,\dots,W_d)|}{|E_{B_i}|} - \frac{|W_i|}{|V_i|}\cdot \frac{|E_i(W_1,\dots,W_d)|}{|E_i|} \right| + \\
+ \frac{|W_i|}{|V_i|} \cdot \left| \frac{|E_i(W_1,\dots,W_d)|}{|E_i|} - \prod_{j=1,j\neq i}^{d}\frac{|W_j|}{|V_j|} \right| = \\
= disc_{B_i}(W_i,E_i(W_1,\dots,W_d)) + \frac{|W_i|}{|V_i|}\cdot disc_{H_i}(W_1,\dots,W_{i-1},W_{i+1},\dots,W_d)
\end{multline}
\end{proof}

\begin{dfn}
A $d$-partite hypergraph $H$ is called type-regular if for any type $i$, $1 \leq i \leq d$, there exist $k_i,l_i\in\mathbb{N}$, 
such that each $i$-type vertex is contained in exactly $k_i$ hyperedges in $H$ and each cotype $i$ wall is contained in exactly $l_i$ hyperedges in $H$.
Note that if $H$ is type-regular, then each induced bipartite graph $B_i$, defined above, is $(k_i,l_i)$-biregular.
\end{dfn}
 
Recall that for a graph $G$ we denote by $\tilde{\lambda}(G)$ the normalized second largest eigenvalue. 
We can now generalize Corollary \ref{bipar} from graphs to hypergraphs.
\begin{cor} \label{disc}
Let $H$ be a $d$-partite type-regular hypergraph. Let $B_i=(V_i,E_i,E_{B_i})$, $i=1\ldots,d$, be the induced bipartite graphs of $H$, as defined above.
Then for any $W_1 \subseteq V_1 ,\ldots, W_d \subseteq V_d$, 
\begin{equation}
disc_H(W_1,\dots, W_d) \leq \sum_{i=1}^{d-1} \left( \tilde{\lambda}(B_i) \cdot \sqrt{\frac{|W_i|}{|V_i|}} \right) .
\end{equation}
In particular
\begin{equation}
\max_{W_i \subseteq V_i} disc_H(W_1,\dots, W_d) \leq (d-1) \cdot \max_{1 \leq i \leq d-1} \tilde{\lambda}(B_i).
\end{equation}
\end{cor}

\begin{proof}
Since $H$ is type-regular, for any type $i \in \mathbb{Z}/d\mathbb{Z}$, there exists a number $l_i \in \mathbb{N}$,
such that any wall of cotype $i$ is contained in exactly $l_i$ facets. Hence,
\begin{equation}
l_i|E_i(W_1,\dots,W_d)| = |E(W_1,\dots,W_{i-1},V_i,W_{i+1},\dots,W_d)| \quad \text{and} \quad l_i|E_i| = |E|,   
\end{equation}
and so,
\begin{equation} \label{eq-329}
\begin{split}
disc_{H_i}(W_1,\dots,W_d)= \left|\frac{|E_i(W_1,\dots,W_d)|}{|E_i|} - \prod_{j=1,j\neq i}^{d}\frac{|W_j|}{|V_j|}\right|=\\
= \left|\frac{|E(W_1,\dots,W_{i-1},V_i,W_{i+1},\dots,W_d)|}{|E|} - \frac{|V_i|}{|V_i|}\prod_{j=1,j\neq i}^{d}\frac{|W_j|}{|V_j|}\right|=\\
=disc_{H}(W_1,\dots,W_{i-1},V_i,W_{i+1},\dots,W_d).
\end{split}
\end{equation}
Corollary \ref{bipar} gives that for any $i=1,\ldots,d$,
\begin{equation}
disc_{B_i}(W_i,E_i(V_1,\ldots,V_{i-1},W_{i+1},\ldots,W_d)) \leq \tilde{\lambda}(B_i) \cdot \sqrt{\frac{|W_i|}{|V_i|}}
\end{equation} 
So, by iterating on Lemma \ref{reduction} and equation \eqref{eq-329}, we get
\begin{equation}
\begin{split}
disc_{H}(W_1,\dots, W_d) \leq \tilde{\lambda}(B_1) \cdot \sqrt{\frac{|W_1|}{|V_1|}} + disc_{H}(V_1,W_2,\dots ,W_d) \leq \\
 \leq \tilde{\lambda}(B_1) \cdot \sqrt{\frac{|W_1|}{|V_1|}} + \tilde{\lambda}(B_2) \cdot \sqrt{\frac{|W_2|}{|V_2|}} + disc_{H}(V_1,V_2,W_3,\dots ,W_d) \leq \\
\ldots \leq \sum_{i=1}^{d-1}\tilde{\lambda}(B_i) \cdot \sqrt{\frac{|W_i|}{|V_i|}} + disc_{H}(V_1,V_2,\dots ,V_{d-1},W_d) =\\
= \sum_{i=1}^{d-1} \tilde{\lambda}(B_i) \cdot \sqrt{\frac{|W_i|}{|V_i|}}
\end{split}
\end{equation}
The last equality follows from the fact that $disc_{H}(V_1,V_2,\dots ,V_{d-1},W_d)=0$, 
since any vertex $w$ of $W_d$ is contained in $k_d$ edges and for elements $w \neq w'$ of $W_d$ these edges are different.
\end{proof}

%%%%%%%%%%%%%%%%%%%%%%%%%%%%%%%%%%%%%%%%%%%%%%%%%%%%%%%%%%%%%%%%%%%%%%%%%%%%%%%%%%%%%%%%%%%%%%%%%%%%%%%%%%%%%%%%%%%%%%%%%%%%%%%%%%%%%%%%%%%%%%%%%%%%%%%%%%%%%%%%%%%%%%%%%%%%%%%%%%
%%%%%%%%%%%%%%%%%%%%%%%%%%%%%%%%%%%%%%%%%%%%%%%%%%%%%%%%%%%%%%%%%%%%%%%%%%%%%%%%%%%%%%%%%%%%%%%%%%%%%%%%%%%%%%%%%%%%%%%%%%%%%%%%%%%%%%%%%%%%%%%%%%%%%%%%%%%%%%%%%%%%%%%%%%%%%%%%%%

\section{Colorful Mixing Lemma for Ramanujan Complexes}\label{sec:Colorful Mixing Lemma}
The goal of this section is to prove the Colorful Mixing Lemma (Theorem \ref{colormix}). 

Let $F$ be a local non-archimedean field whose residue field is of order $q$, and let $\mathcal{B} = \mathcal{B}_d(F),\,d\geq 3,$ be the Bruhat-Tits building of type $\tilde{A}_{d-1}$ associated with $PGL_d(F)$. The building is equipped with a natural $d$-type function which gives it a structure of an infinite $d$-partite hypergraph.
For any cocompact lattice $\Gamma \leq PGL_d(F)$ preserving the type function, the quotient $\mathcal{B}_{\Gamma} = \Gamma \backslash \mathcal{B}_d(F)$ is a finite $d$-partite hypergraph. Recall the notation $dist(\Gamma) = \min_{x \in \mathcal{B},1 \ne \gamma \in \Gamma} \mbox{dist}(\gamma.x,x)$, and that the injectivity radius $r(\Gamma)$ of $\mathcal{B}_{\Gamma}$ is equal to $\lfloor \frac{dist(\Gamma)-1}{2} \rfloor$, since the building $\mathcal{B}$ is its universal cover. 
The Colorful Mixing Lemma reads as follows. Assuming that the injectivity radius of $\mathcal{B}_{\Gamma}$ is at least 2, for any choice of subsets $W_i$ of vertices of $\mathcal{B}_{\Gamma}$ of type $i$:
\begin{equation} \label{discmix}
disc_{\mathcal{B}_{\Gamma}}(W_1,\ldots,W_d) \leq \frac{2d}{q^{1/2}}.
\end{equation}

%%%%%%%%%%%%%%%%%%%%%%%%%%%%%%%%%%%%%%%%%%%%%%%%%%%%%%%%%%%%%%%%%%%%%%%%%%%%%%%%%%%%%%%%%%%%%%%%%%%%%%%%%%%%%%%%%%%%%%%%%%%%%%%%%%%%%%%%%%%%%%%%%%%%%%%%%%%%%%%%%%%%%%%%%%%%%%%%%%
\subsection{The building of type $\tilde{A}_{d-1}$}
In this subsection we review the structure and basic properties of the building $\mathcal{B}_{d}(F)$. Rather than using the general language of buildings, we will present it and prove its properties from basic principles.

Let $F$ be a local non-archimedean field with a discrete valuation $\nu :F^* \rightarrow \mathbb{Z}$, let $\mathcal{O}$ be the ring of integers of $F$, $\pi$ a uniformizer and
$q < \infty$ the cardinality of the residue field $\bar{F} = \mathcal{O} / \pi \mathcal{O}$. For example, $F = \mathbb{F}_q((t))$ the Laurent series, $\nu = \deg$, $\mathcal{O} = \mathbb{F}_q[[t]]$ the Taylor series and $\pi = t$. 

The building $\mathcal{B}=\mathcal{B}_d(F)$ associated to a local field $F$ is an infinite $(d-1)$-dimensional pure simplicial complex constructed as follows.

\paragraph{Vertices.}
A \textit{lattice} is a free $\mathcal{O}$-submodule of $V= F^d$ of rank $d$, i.e. it is of the form $<v_1,\dots,v_d>=\mathcal{O}v_1 + \ldots + \mathcal{O}v_d$ where $\{v_1,\ldots,v_d\}$ is a basis for $V$. Two lattices $L_1,L_2$ are said to be equivalent if there exists $\lambda \in F^*$ such that $L_1 = \lambda L_2$. The equivalence class of a lattice $L$ is denoted by $[L]$. The set of equivalence classes of lattices forms the set of vertices of the building $\mathcal{B}$.

\paragraph{Faces.}
Two vertices $[L_1],[ L_2]$ are connected by an edge in the building, if there exist representatives $L_1' \in [L_1], L_2' \in [L_2]$ such that $\pi L_1' \subset L_2' \subset L_1'$. Note that $L_1' / \pi L_1'$ is a $d$-dimensional vector space over the finite field $\bar{F} = \mathcal{O}/\pi \mathcal{O}$. Fixing a representative $L_1'$ of a vertex $[L_1]$ gives rise to a one-to-one correspondence between the neighbors of $[L_1]$ and the proper subspaces of $L_1' / \pi L_1'$.

A set of vertices $\{[L_1],\ldots,[L_k]\}$ forms a $(k-1)$-face in the building if there exist representatives  $L_i' \in [L_i]$ such that (maybe, after renumbering) $\pi L_1' \subset L_k' \subset \ldots \subset L_2' \subset L_1'$. Note that a $(k-1)$-simplex in the building gives rise to a $k$-flag of subspaces in $L_1' / \pi L_1'$, hence the dimension of the building is $(d-1)$.

The link of every vertex of $\mathcal{B}$ is isomorphic to the flag complex of $\bar{F}^d = \mathbb{F}^{d}_{q}$.

\paragraph{Action of $GL_d(F)$.}
The group $GL_d(F)$ acts transitively on the lattices in $F^d$, and its center preserves the equivalence classes, hence this action induces an action of $G = PGL_d(F)$ on the vertices of the building. The stabilizer of the vertex $[\mathcal{O}^d]$, which is called \textit{the standard lattice}, is $K =PGL_d(\mathcal{O})$, hence the set $G/K$ may be identified with the set of vertices of the building. 

%The Borel subgroup $B$ of $G$ is the subgroup of all upper-triangular matrices, then $G=B\cdot K$ is called the Iwasawa decomposition of $G$. Therefore, the vertices of the building can be identified with the right co-sets $bK$, where $b\in B$.

\paragraph{Type function.} For each vertex $[L]$ there exists an element $g\in G$ such that $[L]=g[\mathcal{O}^d]$. Define the type of $[L]$ to be $\tau([L]) = \nu(\det(g)) \mod d$. It is well defined since for $k\in K$ the determinant $\det(k) \in \mathcal{O}^*$ and hence $\nu(\det(k))=0$. This defines a type function from the vertices of the building to $\mathbb{Z} / d \mathbb{Z} $. Note that a maximal simplex contains vertices of all $d$ types.

For $i\in\mathbb{Z}/d\mathbb{Z}$, denote $G_i = (\nu\circ\det)^{-1}(i)$. Then $G_0$ is the subgroup of $G$ of type-preserving elements, and $G_i$ are its cosets. We saw before that the vertices of the building may be identified with $G/K$. Under this identification, $G_i/K$ is the set of vertices of type $i$.

The group $G_0$ of type-preserving elements is equal to $PSL_d(F) \cdot K$. It is a normal subgroup of $G=PGL_d(F)$ of index $d$, since $G_0 = \{g \in G \mid \nu(\det(g)) = 0 \mod d \}$.

Let now $\Gamma$ be a cocompact discrete subgroup of $G$ which is contained in $G_0$ and $\mathcal{B}_{\Gamma} = \Gamma \backslash \mathcal{B}$.
This is a finite simplicial complex and $\tau$ is well defined on $\mathcal{B}_{\Gamma}$.
The vertices $V$ of $\mathcal{B}_{\Gamma}$ may be identified with $\Gamma \backslash G / K$,  
and in this case the $i$-typed vertices $V_i$ of $\mathcal{B}_{\Gamma}$ are identified with  $\Gamma \backslash G_i / K$.

\paragraph{Relative position.} Define the set $A^{+} = \{a=\overline{(a_1,\ldots,a_d)} \in \mathbb{Z}^d / (1,\ldots,1)\mathbb{Z} \mid  a_1 \leq \ldots \leq a_d \}$, and let $\Lambda^+$ be the set of diagonal matrices in $PGL_d(F)$ of the form $\pi^{a}=\pi^{\overline{(a_1,\ldots,a_d)}} = diag(\pi^{a_1},\ldots,\pi^{a_d})$ for $\overline{(a_1,\ldots,a_d)} \in A^+$. Denote $A_0 = \{\overline{(a_1,\ldots,a_d)} \in A \mid \sum a_i \equiv 0 \mod d\}$.
The Cartan decomposition $G = K \Lambda^+ K$, means that each element $g \in G$ may be written uniquely as $g = k_1\pi^{a} k_2$ for $k_1,k_2 \in K$ and $a \in A^+$. By identifying the vertices of the building with $G/K$, for any two vertices $x=gK$ and $y = hK$ we define the \textit{relative position} of $y$ w.r.t. $x$, 
to be the unique element $a \in A^+$ such that $Kg^{-1}hK = K\pi^{a} K$. 
We get a function $\mathcal{B}(0) \times \mathcal{B}(0) \rightarrow A^+$, where $\mathcal{B}(0)$ is the set of vertices of $\mathcal{B}$.

In other words, for two vertices $x,y$ consider a basis $\{v_1,\ldots,v_d\}$ of $F^d$,
such that $x = [\mathcal{O}v_1 + \ldots + \mathcal{O}v_d]$ and $y = [\pi^{a_1} \mathcal{O}v_1 + \ldots + \pi^{a_d} \mathcal{O}v_d]$ (one can always find such a basis).
Let $g \in G=PGL_d(F)$ be the element which sends the standard basis to $\{v_1,\ldots,v_d\}$, and let $a=\overline{(a_1,\ldots,a_d)} \in A^+$.
Then $x = g.[\mathcal{O}^d]$ and $y = (g\pi^{a}).[\mathcal{O}^d]$, hence $"x^{-1}y" = \pi^{a}$, and the relative position of $y$ w.r.t. $x$ is $a$. We also see that in this case $\tau(x)-\tau(y)\equiv \sum a_i \mod d$, and the relative position of $x$ w.r.t. $y$ is $\overline{(0,a_d-a_{d-1},\dots,a_d - a_1)}$.
In addition, if $y$ is in relative position $\overline{(a_1,\dots,a_d)}$ w.r.t. $x$, then the distance between them, 
i.e. the number of edges in the shortest path connecting them, is equal to $dist(x,y)=a_d - a_1$. 
The action of $G$ on $\mathcal{B}$ preserves the relative position of pairs of vertices.

We note that by the Cartan decomposition, for any $a \in A^+$, $K$ acts transitively on the vertices of a fixed relative position $a$ w.r.t.
the standard lattice $x_0= [\mathcal{O}^d]$. By the transitivity of the action of $G$, for any vertex $x$, $K_x = Stab_G(x)$ acts transitively 
on the vertices of relative position $a$ w.r.t. $x$.

Various combinatorial aspects of the building can be expressed by the relative position:
\begin{lem} \label{neighboor}
Let $y$ be a vertex in the building with relative position $a=\overline{(a_1,\ldots,a_d)}$ w.r.t. $x$. Then:
\begin{enumerate}
\item $x$ and $y$ are neighbors if and only if $a_d = a_1 + 1$, i.e. $a = \overline{(0,\ldots,0,1\ldots,1)}$.
\item $x$ and $y$ are of the same type if and only if $\sum a_i = 0$ modulo $d$, i.e. $a \in A_0$.
\item $x$ and $y$ are separated by a common wall of codimension 1 
(i.e., there exists a $(d-2)$-face $\sigma$ such that $\sigma\cup\{x\}$ and $\sigma\cup\{y\}$ are both $(d-1)$-faces) 
if and only if either $a=\overline{(0,\ldots,0)}$, i.e. they coincide, or $a=\overline{(-1,0,\ldots,0,1)} = \overline{(0,1,\ldots,1,2)}$.
\end{enumerate}
\end{lem}

\begin{proof}
Statements (1) and (2) follow immediately from the definition of the relative position and the discussion above.

To prove (3), assume first that $y$ is in relative position $\overline{(0,1,\dots,1,2)}$ w.r.t. $x$. This implies that $x$ has a representative $L$ with an $\mathcal{O}$-basis $\{v_1,\dots,v_d\}$ such that $L'=<v_1,\pi v_2,\dots,\pi v_{d-1},\pi^2 v_d>$ represents $y$. For $i=1,\dots,d-1$, denote $L_i = <v_1,\dots,v_i,\pi v_{i+1},\dots,\pi v_d>$. Then the set $\sigma = \{z_i=[L_i]\mid 1\leq i\leq d-1\}$ forms a $(d-2)$-cell, and both $\sigma\cup\{x\}$ and $\sigma\cup\{y\}$ are facets of $\mathcal{B}$, since
\begin{equation}
\pi L\subset L_1 \subset \dots L_{d-1}\subset L \quad\mbox{ and }\quad \pi L_{d-1} \subset L'\subset L_1 \subset \dots \subset L_{d-1}.
\end{equation}

To see the opposite direction, assume that there exists a $(d-2)$-cell $\sigma$ with both $\sigma\cup\{x\}$ and $\sigma\cup\{y\}$ being facets of $\mathcal{B}$. As the link of every vertex of $\mathcal{B}$ is the flag complex of $\mathbb{F}^{d}_{q}$, one can decuce that $\sigma$ is contained in $(q+1)$ facets. The fact that $\sigma\cup\{x\}$ is a facet, implies that there exists a representative $L$ of $x$ with an $\mathcal{O}$-basis $\{v_1,\dots,v_d\}$ such that $\sigma = \{z_i=[L_i]\mid 1\leq i\leq d-1\}$, where $L_i = <v_1,\dots,v_i,\pi v_{i+1},\dots,\pi v_d>$.

Here is a list of representatives of $(q+1)$ vertices $\left\{y_\varepsilon = [L'_\varepsilon] \mid \varepsilon\in\mathbb{F}_q\cup\{\infty\}\right\}$ 
such that $\sigma\cup\{y_\varepsilon\}$ is a facet of $\mathcal{B}$:
\begin{equation}
L'_\infty = <\pi v_1,\dots,\pi v_d>\quad (\mbox{note that } y_\infty=[L'_\infty]=x)
\end{equation}
and for $\varepsilon\in\mathbb{F}_q$
\begin{equation}
 L'_\varepsilon = <v_1 + \varepsilon\cdot\pi v_d,\pi v_2,\dots,\pi v_{d-1},\pi^2 v_d>.
\end{equation}
One can easily check that all the $y_\varepsilon$'s are not equivalent and $\sigma\cup\{y_\varepsilon\}$ is a facet, so these are all the facets containing $\sigma$. For every $\varepsilon\in\mathbb{F}_q$, $y_\varepsilon$ is in relative position $\overline{(0,1,\dots,1,2)}$ w.r.t. $x$. This can be seen by taking $\{v_1 + \varepsilon\cdot\pi v_d,v_2,\dots,v_d\}$ as a basis for $L$.
\end{proof}

%%%%%%%%%%%%%%%%%%%%%%%%%%%%%%%%%%%%%%%%%%%%%%%%%%%%%%%%%%%%%%%%%%%%%%%%%%%%%%%%%%%%%%%%%%%%%%%%%%%%%%%%%%%%%%%%%%%%%%%%%%%%%%%%%%%%%%%%%%%%%%%%%%%%%%%%%%%%%%%%%%%%%%%%%%%%%%%%%%
\subsection{Hecke operators}
For any $a=\overline{(a_1,\ldots,a_d)} \in A^+$, define the following Hecke operator on the vertices of the building $H_a : L^2(\mathcal{B}(0)) \to L^2(\mathcal{B}(0))$,
\begin{equation}\nonumber
H_{a} f(xK) = \frac{1}{\mu(K\pi^a K)} \sum_{yK \in xK \pi^a K} f(yK), 
\end{equation}
where $\mu$ is the Haar measure on $G$, normalized such that $\mu(K)=1$, (i.e. $\mu(K\pi^a K) = |K\pi^a K/K|$).
This is the normalized finite sum over the vertices $yK$ of relative position $a$ w.r.t. $xK$. Note that $\mu(K\pi^a K)$ is equal to the number of vertices which are of relative position $a$ w.r.t. $x$.

If $\Gamma$ is a lattice in $G$ with $dist(\Gamma) > a_d - a_1$, then the action of $\Gamma$ commutes with $H_a$. 
Hence, we can consider $H_a$ also as a map $H_a:L^2(V)\to L^2(V)$ where $V$ is the set of vertices of $\Gamma \backslash \mathcal{B}$, so
\begin{equation}\nonumber
H_{a} f(\Gamma x K) = \frac{1}{\mu(K\pi^a K)} \sum_{\Gamma y K \in \Gamma xK\pi^a K} f(\Gamma yK) .
\end{equation}

Moreover if $a \in A_0$, then the type of each $yK \in xK\pi^a K$ is the same as that of $xK$,
so we may consider $H_a$ as a map $H_a:L^2(V_i)\to L^2(V_i)$, where $V_i$ is the set of vertices of type $i$.

Finally, note that if we consider $f \in L^2(\Gamma \backslash G /K)$, 
as a $K$-invariant function in $L^2(\Gamma \backslash G)$, 
and $dk$ is the Haar measure on $K$, normalized such that $dk(K) = 1$,
we may write $H_a$ as an integral over $K$, instead of a sum,
\begin{equation}
H_af(x) = \int_K f(xk\pi^a ) dk.
\end{equation}

Let $(\rho,L^2(\Gamma \backslash G))$ be the unitary $G$-representation, given by right translations $\rho(g) f (x) = f(xg)$.
The following lemma will allow us to give bounds on the spectra of the Hecke operators, assuming we have bounds on the matrix coefficients of the representation 
$L^2(\Gamma \backslash G)$. Bounds on the matrix coefficients will be given at the end of this section, using a theorem by Oh.

\begin{lem} \label{pre-Hecke}
Let $a \in A^+$. For any $K$-invariant vectors $f_1,f_2$ in $L^2(\Gamma \backslash G)$,
\begin{equation}
\langle H_af_1,f_2 \rangle =  \langle \rho(\pi^a)f_1,f_2 \rangle
\end{equation}
\end{lem}

\begin{proof}
\begin{equation}
\begin{split}
\langle H_af_1,f_2 \rangle = \int_{\Gamma \backslash G} H_a\left(f_1(x)\right)\overline{f_2(x)}dx = \\
= \int_{\Gamma \backslash G} \left(\int_K f_1(xk\pi^a )dk\right) \overline{f_2(x)}dx = \\
= \int_K \left(\int_{\Gamma \backslash G} f_1(xk\pi^a )\overline{f_2(x)}dx\right)dk,
\end{split}
\end{equation}
where the last equality follows from Fubini's theorem.
Since the measure $dx$ is right invariant we can replace $x$ by $xk$, and by the $K$-invariance of $f_2$ we get
\begin{equation}\nonumber
\begin{split}
\langle H_af_1,f_2 \rangle =  \int_K \left(\int_{\Gamma \backslash G} f_1(x\pi^a )\overline{f_2(x)}dx\right)dk =\\
= \int_{\Gamma \backslash G} f_1(x\pi^a )\overline{f_2(x)}dx = \langle \rho(\pi^a )f_1 , f_2 \rangle.
\end{split}
\end{equation}
\end{proof}

\begin{lem} \label{G^+K=G_0}
Let $i\in\mathbb{Z}/d\mathbb{Z}$ be a type and $f \in L^2_0(V_i)$, where $L^2_0(V_i)=\{g\in L^2(V_i)\mid \langle g, \mathds{1}_{V_{i}}\rangle=0\}$. 
Extend $f$ to a function in $L^2_0(V)$ by setting it to be zero outside $V_i$.
Recall that $V = \Gamma \backslash G /K$, where $K=PGL_d(\mathcal{O})$, and consider $f$ as a $K$-invariant vector in $L^2(\Gamma \backslash G)$.
Then $f$ is orthogonal to any $G^+$-invariant vector, where $G^+=PSL_d(F)$.
\end{lem}

\begin{proof}
Let $h \in L^2(\Gamma \backslash G)$ be a $G^+$-invariant function. Define $\tilde{h}(x) = \int_Kh(xk)dk$. 
Recall that the type-preserving subgroup $G_0$ is equal to $ G^+K$.
So, $\tilde{h}$ is $G^+$-invariant and $K$-invariant, and hence $G_0$-invariant.
Since each $G_i$ is a $G_0$-coset, $\tilde{h}$ is constant on each $\Gamma \backslash G_i$ and hence on $V_i = \Gamma \backslash G_i/K$.
Now, since $f \in L^2_0(V_i)$ we get that 
\[ \langle f , \tilde{h} \rangle = 0.\]
On the other hand, 
\begin{multline} \nonumber
\langle f ,h \rangle = \int_{\Gamma \backslash G} f(x)\overline{h(x)}dx = \int_K\left(\int_{\Gamma \backslash G} f(x)\overline{h(x)}dx\right)dk = \\
= \int_K\left(\int_{\Gamma \backslash G} f(xk^{-1})\overline{h(x)}dx\right)dk = \int_K\left(\int_{\Gamma \backslash G} f(y)\overline{h(yk)}dy\right)dk = \\
= \int_{\Gamma \backslash G} f(y)\overline{\left(\int_K h(yk)dk\right)}dy = \int_{\Gamma \backslash G} f(y)\overline{\tilde{h}(y)}dy = \langle f ,\tilde{h} \rangle,
\end{multline}
where again we used: $\int_Kdk =1$, the $K$-invariance of $f$, the Haar measure $dx$ being right invariant and Fubini's Theorem, respectively.
So $f$ is orthogonal to $h$, which proves the claim.
\end{proof}

Recall that $A_0=\{\overline{(a_1,\ldots,a_d)} \in A | \sum a_i  \equiv 0 (\mbox{mod }d)\}$, is the set of type-preserving translations (see Lemma \ref{neighboor}).
So, for any $a \in A_0$ the Hecke operator $H_a$ is a well defined operator from $L^2(V_i)$ to itself, for any type $i$.

Combining Lemmas \ref{pre-Hecke} and \ref{G^+K=G_0} we get the following bound on the norm of the Hecke operator in terms of the matrix coefficients.
\begin{cor} \label{Hecke}
For any type $i\in\mathbb{Z}/d\mathbb{Z}$ and any $a \in A_0$,
\begin{equation}
\|H_a\|_{L^2_0(V_i)} \leq \sup_{f_1,f_2} \langle \rho(\pi^a)f_1,f_2 \rangle
\end{equation}
where $f_1,f_2$ run over all the $K$-invariant normalized vectors in $L^2(\Gamma \backslash G_i)$ orthogonal to any $G^+$-invariant vector 
(when considered as functions in $L^2(\Gamma\backslash G)$ as in Lemma \ref{G^+K=G_0}.)
\end{cor}

\begin{proof}
Let $f_1,f_2 \in L^2_0(V_i)$ of norm $1$ be such that $\|H_a\|_{L^2_0(V_i)} = \langle H_af_1,f_2 \rangle$. By Lemma \ref{pre-Hecke},
\[\|H_a\|_{L^2_0(V_i)}  = \langle \rho(\pi^a)f_1,f_2 \rangle\]
and by Lemma \ref{G^+K=G_0}, $f_1,f_2$ are orthogonal to any $G^+$-invariant vector in $L^2(\Gamma \backslash G)$, which proves the claim.
\end{proof}

%%%%%%%%%%%%%%%%%%%%%%%%%%%%%%%%%%%%%%%%%%%%%%%%%%%%%%%%%%%%%%%%%%%%%%%%%%%%%%%%%%%%%%%%%%%%%%%%%%%%%%%%%%%%%%%%%%%%%%%%%%%%%%%%%%%%%%%%%%%%%%%%%%%%%%%%%%%%%%%%%%%%%%%%%%%%%%%%%%
\subsection{Adjacency operators}

Again let $\Gamma$ be a cocompact lattice in $G$ with $\Gamma\subseteq G_0$, i.e. $\Gamma$ preserves the type function. 
So $\mathcal{B}_\Gamma = \Gamma\backslash \mathcal{B}$ is a $d$-partite hypergraph.

%adjacency operator of the induced bipartite biregualr graphs of the building
Recall that for each type $i\in\mathbb{Z}/d\mathbb{Z}$, the induced bipartite graph $B_i$ of the $d$-partite hypergraph $\mathcal{B}_{\Gamma}$, 
has the $i$-type vertices $V_i$ on one side, and the walls $E_i$, i.e. the simplicies of dimension $(d-2)$ of cotype $i$, on the other side. 
A vertex and a wall are connected if their union forms a maximal simplex in $\mathcal{B}_\Gamma$.

Let $\widetilde{A} = \widetilde{A}(B_i)$ be the normalized adjacency operator of $B_i$, i.e. 
\[ \widetilde{A}f(x) := \frac{1}{\deg(x)}\sum_{y \sim x}f(y), \]
where the summation is over all the neighbors of $x$ in $B_i$. In the natural basis the matrix of $\widetilde{A}$ is a block matrix of the form 
\begin{equation}\nonumber
\widetilde{A} = \left( \begin{array}{cc} 0 & N \\ N^t & 0 \end{array} \right),
\end{equation}
where $N$ is a matrix of size $|V_i|\times |E_i|$.

Define $D_i$ to be the multigraph on $V_i$, where two vertices are connected by as many edges as there are paths of length 2 in the graph $B_i$ connecting them. 
Then the matrix $N N^t$ is the matrix of the normalized adjacency operator of the multigraph $D_i$. 
Note that the number of loops on each vertex in $D_i$ is equal to the vertex degree in $B_i$.

The non-zero eigenvalues of the matrices $N N^t$ and $N^t N$ coincide, 
and $\lambda \ne 0$ is an eigenvalue of $N N^t$ if and only if $\sqrt{\lambda}$ is an eigenvalue of $\widetilde{A}$.
So, in order to bound the eigenvalues of $\widetilde{A}$, it is enough to bound the eigenvalues of $NN^t$. 
 
\begin{lem} \label{adj}
The operator $NN^t$, as an operator from $L^2(V_i)$ to itself, is a convex sum of two Hecke operators $I = H_{(0,\ldots,0)}$ and $H_{(-1,0,\ldots,0,1)}$, in fact,
\begin{equation}\nonumber
NN^t = \frac{1}{q+1} I + \frac{q}{q+1} H_{(-1,0,\ldots,0,1)}.
\end{equation}
\end{lem}

\begin{proof}
By Lemma \ref{neighboor}, two vertices of type $i$ of the building share a common wall if and only if their relative position is either 
$\overline{(0,\ldots,0)}$ or $\overline{(-1,0,\ldots,0,1)}$,
i.e. a vertex $xK$ of type $i$ shares a common wall with vertices which are either the right $K$-cosets in $xK\pi^{(-1,0,\ldots,0,1)}K$ or $xK$ itself.

In the quotient $\mathcal{B}_{\Gamma}$, the vertex $\Gamma x K$ of type $i$ can be lifted to the vertex $xK$ in the building. The $i$-type vertices in the building which
share a common wall with $xK$ are mapped surjectively to the vertices in the quotient which share a common wall with $\Gamma xK$ in the quotient. 
Since $\Gamma$ has injectivity radius $\geq 2$, this map is also injective.

Hence, after the normalization, by the definition of the Hecke operators we get, that 
\begin{equation}\nonumber
NN^t = c_{(0,\ldots,0)} H_{(0,\ldots,0)} + c_{(-1,0,\ldots,0,1)} H_{(-1,0,\ldots,0,1)},
\end{equation}
where $c_{a}$ is the number of edges in $D_i$ connecting a vertex $xK$ to vertices of relative position $a$ with respect to it,
 divided by the degree of the vertex $xK$ in the graph $D_i$. Clearly $c_{(0,\ldots,0)} + c_{(-1,0,\ldots,0,1)} =1$.

Each wall of the building $\mathcal{B}_d(F)$ is contained in exactly $q+1$ chambers, and each $i$-type vertex is contained in exactly $r$ chambers (= facets), 
where the number $r$ depends on $d$ and $q$, but not on the vertex. In the quotient $\mathcal{B}_{\Gamma}$, since the injectivity radius $\geq 2$, 
each wall is also contained in exactly $q+1$ chambers and each $i$-type vertex is contained in $r$ chambers. Hence $B_i$ is a bipartite $(r,q+1)$-biregular graph. 
Therefore $D_i$ is a $r(q+1)$ regular multi-graph, where each vertex has exactly $r$ loops, so $c_{(0,\ldots,0)} = \frac{1}{q+1}$, which completes the proof.
\end{proof}

We can now get the following bound on the normalized second largest eigenvalue of $B_i$,
in terms of matrix coefficients of a certain unitary representation.
\begin{cor} \label{2ev-bound}
Let $(\rho,L^2(\Gamma \backslash G))$ be the unitary $G$-representation given by right translation $\rho(g)f(x) = f(xg)$.
Let $W \leq L^2(\Gamma \backslash G)$ be the subspace of $K$-invariant vectors which are orthogonal to all $G^+$-invariant vectors.
For any $g \in G$, define $\rho_W(g)$ to be the maximal absolute value of a matrix coefficient of normalized vectors from $W$ on the element $g$,
i.e. 
$$ \rho_W(g) = \sup_{f_1,f_2 \in W,\|f_1\|=\|f_2\|=1} |\langle \rho(g)f_1,f_2 \rangle|.$$

Then the normalized second largest eigenvalue of $B_i$, $\tilde{\lambda_i} = \tilde{\lambda}(B_i)$, satisfies
\begin{equation}
\tilde{\lambda_i} \leq \sqrt{\frac{1}{q+1} + \frac{q}{q+1}\rho_W(\pi^{(-1,0,\ldots,0,1)})} \leq \sqrt{q^{-1} + \rho_W(\pi^{(-1,0,\ldots,0,1)})}.
\end{equation}
\end{cor}

\begin{proof}
By Lemma \ref{adj} and the discussion before it,
\begin{equation}\nonumber
\tilde{\lambda_i} = \sqrt{ \|NN^t \|_{L_0^2(V_i)}} \leq \sqrt{\frac{1}{q+1} +  \frac{q}{q+1} \|H_{(-1,0,\ldots,0,1)} \|_{L_0^2(V_i)}},
\end{equation}
and, by Corollary \ref{Hecke}, $\|H_{(-1,0,\ldots,0,1)} \|_{L_0^2(V_i)} \leq \rho_W(\pi^{(-1,0,\ldots,0,1)})$.
\end{proof}

%%%%%%%%%%%%%%%%%%%%%%%%%%%%%%%%%%%%%%%%%%%%%%%%%%%%%%%%%%%%%%%%%%%%%%%%%%%%%%%%%%%%%%%%%%%%%%%%%%%%%%%%%%%%%%%%%%%%%%%%%%%%%%%%%%%%%%%%%%%%%%%%%%%%%%%%%%%%%%%%%%%%%%%%%%%%%%%%%%
\subsection{Proof of the mixing Lemma}
The following result by Oh, gives a unified bound on the matrix coefficients of a unitary representation of a reductive group over a local field.

%Oh's theorem
\begin{thm}\cite[Theorem~1.1]{Oh} \label{oh}
Let $F$ be a local non-archimedean field with $char(F) \ne 2$.
Let $G$ be the group of the $F$-rational points of an $F$-split connected reductive group of rank $\geq 2$ and $G/Z(G)$ almost $F$-simple.
Let $G^+$ be the the subgroup of $G$ generated by the unipotent elements of $G$.

Let $\Phi$ be a root system of $G$ with regard to some maximal torus $T$, and $\Phi^+$ the set of positive roots in $\Phi$.
Let $S \subset \Phi^=$ be a strongly orthogonal system of roots, which, by definition, means $\forall \alpha,\beta \in S \, \Rightarrow \alpha \pm \beta \not \in S$.

Let $K$ be a good maximal compact subgroup of $G$, which means that $K$ is a stabilizer of a special vertex in the building of $G$.
Any good maximal compact subgroup gives rise to a Cartan decomposition $G = K \Lambda^+ K$, where $\Lambda^+$ is a positive Weyl chamber.

Then for any unitary representation $\rho$ of $G$ without a non-zero $G^+$-invariant vectors, 
and for any $K$-finite unit vectors $v$ and $u$, 
\begin{equation}
|\langle \rho(g)v,u \rangle| \leq (\dim(Kv) \dim(Ku))^{\frac{1}{2}} \xi_S(\lambda)
\end{equation}
where $g=k_1 \lambda k_2\in K \Lambda^+ K = G$, $\xi_S(\lambda) = \prod_{\alpha \in S} \Xi_{PGL_2(F)}\left( \begin{array}{cc} \alpha(\lambda) & 0 \\ 0 & 1 \end{array} \right)$ and $\Xi_{PGL_2(F)}$ is the Harish-chandra $\Xi$-function of $PGL_2(F)$.
\end{thm}

In our case, $G = PGL_d(F)$ satisfies the assumptions of Theorem \ref{oh} and $G^+$ is equal to $PSL_d(F)$.
The subgroup $K$ is the stabilizer of the fundamental lattice, hence $K$ is a good maximal compact subgroup.
As a maximal torus $T$, we take the subgroup of diagonal matrices, and as strongly orthogonal system we take the singelton $S =\{\alpha := a_d - a_1\}$, 
$\alpha(diag(\pi^{a_1}b_1,\ldots,\pi^{a_d}b_d)) := \pi^{a_d - a_1}$, where $b_1,\dots,b_d \in \mathcal{O}^*$.

Using the following formula (see \cite[Section~3.8]{Oh}), for $n \in \mathbb{N}$,
\begin{equation}
\Xi_{PGL_2(F)} \left( \begin{array}{cc} \pi^{\pm n} & 0 \\ 0 & 1 \end{array} \right) 
= q^{-n/2}\left(\frac{n(q-1)+q+1}{q+1}\right) \leq (n+1)q^{-n/2}
\end{equation}

We get the following application of Oh's theorem, when this time $G=PGL_d(F)$:
\begin{cor} \label{Oh}
Define $(\rho,L^2(\Gamma \backslash G))$ to be the $G$-representation given by right translation $\rho(g)f(x) = f(xg)$.
Let $f,f' \in L^2(\Gamma \backslash G) $ be $K$-invariant unit vectors which are orthogonal to all $G^+$-invariant vectors.
Then for $g=\pi^{(a_1,\dots,a_d)}$,
\begin{equation}
|\langle \rho(g)f,f' \rangle| \leq (a_d-a_1+1) q^{-\frac{a_d-a_1}{2}}.
\end{equation}
\end{cor}

Combining all these estimates together with Corollary \ref{disc}, we can prove the colorful mixing lemma.
\begin{proof}[Proof of the colorful mixing lemma.]
Let $\tilde{\lambda_i}= \tilde{\lambda}(B_i)$ be the normalized second largest eigenvalue of the bipartite graph $B_i$. 
By Corollary \ref{2ev-bound}
\[ \tilde{\lambda_i} \leq \sqrt{q^{-1} + \rho_W(\pi^{(-1,0,\ldots,0,1)})} \]
By Corollary \ref{Oh} in the notation of Corollary \ref{2ev-bound} we have
\[\rho_W(\pi^{(-1,0,\ldots,0,1)}) \leq 3q^{-1}\]
Combining these together, we get that for any type $i$,
\[\tilde{\lambda}(B_i) \leq \frac{2}{q^{1/2}}. \]
Finally, by Corollary \ref{disc}, for any choice of sets $W_i \subset V_i,i=1,\dots,d$,
\[ disc_{\Gamma \backslash \mathcal{B}}(W_1,\ldots,W_d) \leq d \cdot \max_i \tilde{\lambda}(B_i) \leq \frac{2d}{q^{1/2}}, \]
which proves the claim.
\end{proof}

%%%%%%%%%%%%%%%%%%%%%%%%%%%%%%%%%%%%%%%%%%%%%%%%%%%%%%%%%%%%%%%%%%%%%%%%%%%%%%%%%%%%%%%%%%%%%%%%%%%%%%%%%%%%%%%%%%%%%%%%%%%%%%%%%%%%%%%%%%%%%%%%%%%%%%%%%%%%%%%%%%%%%%%%%%%%%%%%%%
%%%%%%%%%%%%%%%%%%%%%%%%%%%%%%%%%%%%%%%%%%%%%%%%%%%%%%%%%%%%%%%%%%%%%%%%%%%%%%%%%%%%%%%%%%%%%%%%%%%%%%%%%%%%%%%%%%%%%%%%%%%%%%%%%%%%%%%%%%%%%%%%%%%%%%%%%%%%%%%%%%%%%%%%%%%%%%%%%%

\section{Explicit construction of Ramanujan complexes}\label{sec:Explicit construction of Ramanujan complexes}
Ramanujan complexes were introduced in \cite{Li}, \cite{LSV1} and \cite{Sar} as a generalization of the Ramanujan graphs constructed in \cite{LPS}, and were explicitly constructed in \cite{LSV2}. These complexes are certain quotients of the Bruhat-Tits buildings. 

The heart of the construction in \cite{LSV2} is the Cartwright-Steger lattice (CS-lattice) $\Lambda$ \cite{CS}, 
which allows us to view some of the quotients of the building as Cayley complexes of finite groups with explicit sets of generators.

The reader is referred to \cite{LSV1,LSV2} for more details, and to \cite{Lu2} for a reader friendly survey.

%%%%%%%%%%%%%%%%%%%%%%%%%%%%%%%%%%%%%%%%%%%%%%%%%%%%%%%%%%%%%%%%%%%%%%%%%%%%%%%%%%%%%%%%%%%%%%%%%%%%%%%%%%%%%%%%%%%%%%%%%%%%%%%%%%%%%%%%%%%%%%%%%%%%%
\subsection{The Cartwright-Steger lattice}
Here we present the CS-lattice, and express explicitly its set of generators.

%The division algebra 
Let $\mathbb{F}_q$ be the finite field of size $q$, and $\mathbb{F}_{q^d}$ the field extension of $\mathbb{F}_q$ of degree $d$.
Let $\phi$ be a generator of the Galois group $Gal(\mathbb{F}_{q^d}/\mathbb{F}_q) \cong \mathbb{Z}/ d\mathbb{Z}$, and fix a basis $\xi_0,\ldots,\xi_{d-1}$ of $\mathbb{F}_{q^d}$
over $\mathbb{F}_q$ with $\xi_i = \phi^i(\xi_0)$. Denote $R_T = \mathbb{F}_q[y,\frac{1}{1+y}]$. For a given $R_T$-algebra $S$ (i.e. $S$ is given with a ring homomorphism 
$R_T \rightarrow S$), we define a $S$-algebra $\mathcal{A}(S) = \bigoplus_{i,j=0}^{d-1}S\xi_i z^j$ with the relations $z^d = 1+y$ and $z \xi_i = \phi(\xi_i)z$. 
One can see that the center of $\mathcal{A}(S)$ is $S$, and $\mathcal{A}(S)^*/S^*$ is a group scheme for $R_T$-algebras.

%The arithmetic lattice
Let $k=\mathbb{F}_q(y)$. Then $\mathcal{A}(k)$ is a $k$-central simple algebra.
For almost all completions $k_{\nu}$ of $k$, $\mathcal{A}(k_{\nu})$ splits, i.e. $\mathcal{A}(k_{\nu}) \cong M_d(k_{\nu})$. 
In fact, this happens for all completions except for $\nu_{\frac{1}{y}}$ and $\nu_{1+y}$. 
In particular, for $F = \mathbb{F}_q((y)) = k_{\nu_y}$, the algebra splits and $\mathcal{A}(F)^*/F^* \cong PGL_d(F)$ (see \cite[Proposition~3.1]{LSV2}). 
On the other hand, for $\nu = \nu_{\frac{1}{y}}$ or $\nu_{1+y}$, $\mathcal{A}(k_{\nu})$ is a division algebra and $\mathcal{A}(k_{\nu})^*/k_{\nu}^*$ is compact. 
Thus $\mathbb{F}_q[\frac{1}{y},y,\frac{1}{1+y}] \hookrightarrow k_{\nu_{y}} \times k_{\nu_{\frac{1}{y}}} \times k_{\nu_{1+y}}$ is discrete and by substituting these rings in $\mathcal{A}(-)^*/(-)^*$ and projecting to the first coordinate $\mathcal{A}(F)^*/F^* \cong PGL_d(F)$ we get a discrete subgroup, which is an arithmetic lattice. 

%The CS lattice
Denote $R = \mathbb{F}_q[y,\frac{1}{y},\frac{1}{1+y}]$. As $1+y$ is invertible in $R_T$, $z$ is invertible in $\mathcal{A}(R_T)$, since it divides $z^d = 1+y$. 
Denote $b = 1 - z^{-1} \in \mathcal{A}(R_T)$. Since $y$ is invertible in $R$, so is $\frac{y}{1+y}$, and hence $b$ is invertible in $\mathcal{A}(R)$, 
since it divides $1 - z^{-d} = \frac{y}{1+y}$. 
For an element $u \in \mathbb{F}^*_{q^d} \subset \mathcal{A}(R)$, denote $b_u = u b u^{-1}$ and note that as $\mathbb{F}_q \subset R$ is in the center of $\mathcal{A}(R)$, $b_u$ depends only on the coset $u \in \mathbb{F}^*_{q^d}/\mathbb{F}^*_q$. Define $\Sigma_1 = \{\bar{b}_u \, |\, u \in \mathbb{F}^*_{q^d}/\mathbb{F}^*_q  \} \subset \mathcal{A}(R)^*/R^* \subset \mathcal{A}(F)^*/F^* \cong PGL_d(F) $ where $F=\mathbb{F}_q((y))$. The Cartwright-Steger lattice is $\Lambda = \langle \Sigma_1  \rangle$. 
This is the promised CS-lattice which acts simply transitively on the vertices of the building $\mathcal{B} =\mathcal{B}_d(\mathbb{F}_q((y)))$ (see \cite{CS} and \cite[Proposition~4.8]{LSV2}).

%The Cayley Complex
More explicitly, let $x_0=[\mathcal{O}^d]$ be the vertex of the building corresponding to the standard lattice, 
and let $\tau : \mathcal{B} \rightarrow \mathbb{Z}/d\mathbb{Z}$ be the type function on the building. 
Then for each neighboring vertex $x$ of $x_0$ with $\tau(x) = 1$, there exists a unique $b_u \in \Sigma_1$ such that $b_u\cdot x_0 = x$.
Now, for $i=2,\dots,d-1$ denote $N_i = \{x \in V(\mathcal{B})| \; x \sim x_0 \mbox{ and } \tau(x)=i \}$. For every $x \in N_i$, 
there is a unique $\gamma_x \in \Lambda$ with $\gamma_x.x_0 = x$. Let $\Sigma_i = \{\gamma_x | x \in N_i\}$, so $\Sigma_i.x_0 = N_i$.
Let $\Sigma = \cup_{i=1}^{d-1} \Sigma_i$. Then $\Sigma .x_0$ is the set of all the neighbors of $x_0$, 
and we can identify the 1-skeleton of the building with the Cayley graph $Cay(\Lambda,\Sigma)$. 
Note that as $\Lambda$ acts simply transitively, $|\Sigma_i| = {d \brack i}_q$ and hence $|\Sigma| = \sum_{i=1}^{d-1}{d \brack i}_q$, 
where ${d \brack i}_q$ is the number of $i$-dimensional subspaces of a $d$-dimensional vector space over $\mathbb{F}_q$.   

Recall that a clique in a graph is a set of vertices such that each pair of them is connected by an edge, 
and the clique complex of a graph is defined to be the collection of the cliques in the graph. The clique complex of a Cayley graph is called a Cayley complex.
The building $\mathcal{B}$ and its quotients with injectivity radius $\geq 2$ are clique complexes, hence completely determined by their 1-skeleton. 

We can conclude that if $\Gamma \lhd \Lambda$, the complex $\Gamma \backslash \mathcal{B}$ is the Cayley complex $Cay(\Lambda / \Gamma , \Sigma)$
of the quotient group $\Lambda / \Gamma$ w.r.t. the set of generators $\Sigma$.
In the next subsection we will apply this in the case where $\Gamma$ is a normal congruence subgroup of $\Lambda$.

%%%%%%%%%%%%%%%%%%%%%%%%%%%%%%%%%%%%%%%%%%%%%%%%%%%%%%%%%%%%%%%%%%%%%%%%%%%%%%%%%%%%%%%%%%%%%%%%%%%%%%%%%%%%%%%%%%%%%%%%%%%%%%%%%%%%%%%%%%%%%%%%%%%%%
\subsection{Congruence subgroups}
Ramanujan complexes are obtained in \cite{LSV2} by dividing the building modulo the action of congruence subgroups of some arithmetic cocompact lattices,
such as $\Lambda$ above. Here we define the congruence subgroups of $\Lambda$ and display their quotients as Cayley complexes of some finite groups.

%the congruence subgroup
For an ideal $0 \ne I \lhd R$, define the congruence subgroup of $\Lambda$ to be 
$\Lambda(I) = \Lambda \cap \ker(\mathcal{A}(R)^*/R^* \rightarrow \mathcal{A}(R/I)^*/(R/I)^*)$.
This congruence subgroup is a finite index normal subgroup of $\Lambda$.
Hence the quotient $\Lambda(I) \backslash \mathcal{B}$ is a finite simplicial complex, which we will identify with the Cayley complex of the group $\Lambda / \Lambda(I)$
(w.r.t. $\Sigma$ as the set of generators). By \cite[Theorem~6.2]{LSV2} for any $0 \ne I \lhd R$, $\Lambda / \Lambda(I)$ is a Ramanujan complex
(though, in this paper, we are not really using this deep fact).

%Finite groups
By \cite[Theorem~6.6]{LSV2}, the group $\Lambda / \Lambda(I)$ can be identified as a subgroup of $PGL_d(R/I)$ which contains $PSL_d(R/I)$. As $R = \mathbb{F}_q[y,\frac{1}{y},\frac{1}{1+y}]$, we consider $I = (f)$ where $ f \in \mathbb{F}_q[y] $ is an irreducible polynomial of degree $e \geq 2$. Then $R/I \cong \mathbb{F}_{q^e}$ and hence $PSL_d(\mathbb{F}_{q^e}) \leq \Lambda / \Lambda(f) \leq PGL_d(\mathbb{F}_{q^e})$. Moreover, by \cite[Theorem~7.1]{LSV2}, (assuming $q^e > 4 d^2+1$), for any subgroup $PSL_d(\mathbb{F}_{q^e}) \leq G \leq PGL_d(\mathbb{F}_{q^e})$, we may find $f$ such that $G = \Lambda / \Lambda(f)$.
In particular, $G$ has a set of $\sum_{i=1}^{d-1}{d \brack i}_q$ generators such that the corresponding Cayley complex is a Ramanujan complex.

%injectivity radius
For such quotients of $\mathcal{B}$ by congruence subgroups, a bound on the injectivity radius was presented in \cite{LuMe}:

\begin{pro}\cite[Proposition~3.3]{LuMe}
Let $ f \in \mathbb{F}_q[y] $ be an irreducible polynomial, $\Gamma = \Lambda(f)$, and $X = \Gamma \backslash \mathcal{B}$.
Denote by $|X|$ the number of vertices of $X$. Then
\[dist(\Gamma):= \min_{1 \ne \gamma \in \Gamma, x \in \mathcal{B}}dist(\gamma.x,x) \geq \frac{\deg(f)}{d} \geq \frac{\log_q|X|}{(d-1)(d^2-1)}\]
\end{pro}

As a consequence of this proposition, we get
\begin{cor}[injectivity radius] \label{inj-rad}
Let $X= \Gamma \backslash \mathcal{B}$ be a quotient of the building, where $\Gamma$ as above. Then the injectivity radius $r(X)$ of $X$ satisfies 
\[r(X) \geq \lfloor \frac{dist(\Gamma) -1}{2} \rfloor \geq \frac{\log_q|X|}{2(d-1)(d^2-1)} - \frac{1}{2}.\]
\end{cor}

%%%%%%%%%%%%%%%%%%%%%%%%%%%%%%%%%%%%%%%%%%%%%%%%%%%%%%%%%%%%%%%%%%%%%%%%%%%%%%%%%%%%%%%%%%%%%%%%%%%%%%%%%%%%%%%%%%%%%%%%%%%%%%%%%%%%%%%%%%%%%%%%%%%%%
\subsection{Partite Ramanujan complexes}

%Ramanujan graphs
Before continuing, let us examine the special case of $d=2$, i.e. the Ramanujan graphs (compare with \cite{LPS} and \cite{Lu1}). 
In this case $\Lambda / \Lambda(I)$ is a subgroup of $PGL_2(\mathbb{F}_{q^e})$ containing $PSL_2(\mathbb{F}_{q^e})$, 
and since $PSL_2(\mathbb{F}_{q^e})$ is of index $2$ inside $PGL_2(\mathbb{F}_{q^e})$, $\Lambda / \Lambda(I)$ is either $PGL_2(\mathbb{F}_{q^e})$ or $PSL_2(\mathbb{F}_{q^e})$. 
In this case all the elements of $\Sigma = \Sigma_1$ either lie outside of $PSL_2(\mathbb{F}_{q^e})$ or all are inside of it 
(which is the case iff the image of $b$ is in it, in which case all the $b_u$, which are conjugates of $b$, are also there). 
The Cayley graph $Cay(\Lambda / \Lambda(I),\Sigma)$ is bipartite in the first case and has large chromatic number in the second. 
In other words, the quotient $\Lambda / \Lambda(I)$ inherits the type function of the building ($\tau : \mathcal{B} \rightarrow \mathbb{Z}/ 2\mathbb{Z}$) 
if and only if the index of $PSL_2(\mathbb{F}_{q^e})$ inside $\Lambda / \Lambda(I)$ is $2$, i.e. $\Lambda / \Lambda(I) = PGL_2(\mathbb{F}_{q^e})$.

%Ramanujan partite complexes
In the high-dimensional case the situation is in analogy with the 1-dimensional case (see  \cite[Proposition~6.7,~Corollary~6.8]{LSV2}).
Assume $I=(f)$ as before with $R/I=\mathbb{F}_{q^e}$. Then $PGL_d(\mathbb{F}_{q^e})/PSL_d(\mathbb{F}_{q^e}) \cong \mathbb{Z}/d\mathbb{Z}$.
If $r$ is the index of $PSL_d(R/I)$ in $\Lambda / \Lambda(I)$ then $r|d$, and the image of $\Lambda$ in $\mathbb{Z}/d\mathbb{Z}$ under the map $\tau$ is 
$\frac{d}{r}\mathbb{Z}/d\mathbb{Z} \cong \mathbb{Z}/r\mathbb{Z}$.
The quotient then inherits an $r$-partition from the building, i.e. $\tau : \Lambda / \Lambda(I) \rightarrow  \mathbb{Z}/ r\mathbb{Z}$. 
In the construction above, the index $r = [\Lambda / \Lambda(I) : PSL_d(R/I)]$ is equal to the order of $\frac{y}{1+y}$ inside $(R/I)^*/(R/I)^{*d}$ 
(see \cite[Proposition~6.7]{LSV2}).

\begin{lem} \label{cover}
Let $0 \ne I \lhd R$ and $\Gamma = \Lambda(I)$ as above, $r=[\Lambda /\Gamma : PSL_d(R/I)]$, and consider the simplicial complex $\Gamma \backslash \mathcal{B}$.
\begin{itemize}
\item[(a)] Denote $\Gamma_0 = \Gamma \cap G_0 = \{g \in \Gamma \mid \nu_F(\det(g))  \equiv 0 \mbox{ mod }d\}$ the subgroup of type-preserving elements of $\Gamma$.
Then $[\Gamma : \Gamma_0] = \frac{d}{r}$.
\item[(b)] If $r=1$, then $\Gamma_0 \backslash \mathcal{B} \rightarrow \Gamma \backslash \mathcal{B}$ is a $d$-cover.
Moreover, for each vertex in $\Gamma \backslash \mathcal{B}$, its preimage is a set of $d$ vertices, one of each type in $\mathbb{Z}/d\mathbb{Z}$.
\end{itemize}
\end{lem}

\begin{proof}
Let us look at the type function as a surjective homomorphism of groups 
\[\tau = \nu_F \circ \det : \Lambda \cong \mathcal{B}^{(0)} \rightarrow \mathbb{Z}/d\mathbb{Z}  \]
If $\Lambda / \Gamma$ is an extension of the simple group $PSL_d(q^e)$ by a cyclic group of order $r$, it follows that by restricting the homomorphism, 
we get a surjective homomorphism, $\tau_{\Gamma} : \Gamma \rightarrow \mathbb{Z}/\frac{d}{r}\mathbb{Z} \cong r\mathbb{Z}/d\mathbb{Z}$.
Now, since $G_0$ is the subgroup of type-preserving elements in $G$, then $\ker(\tau|_X) = X \cap G_0$.
So by the First Isomorphism Theorem we have 
\[\Gamma / \Gamma_0 = \Gamma / \ker(\tau_{\Gamma}) \cong \tau(\Gamma) \cong \mathbb{Z}/\frac{d}{r}\mathbb{Z}.\]
This proves (a). To prove (b) we argue as follows:

Since $r=1$ then by (a) we have that $[\Gamma : \Gamma_0] = d$. Let $\gamma_1,\ldots,\gamma_d$ be representatives of $\Gamma_0$-cosets.
Since $\Gamma_0 = \ker(\nu_F \circ \det|_{\Gamma})$ all the $d$ types are obtained, then after renumbering, for each $i \in \mathbb{Z}/d\mathbb{Z}$, 
$\nu_F \circ \det(\gamma_i) = i$.
Also, for any $\Gamma x \in \Gamma \backslash \mathcal{B}$, its preimages are $\Gamma_0 \gamma_1 x,\ldots,\Gamma_0 \gamma_d x$,
and their types are $1+\tau(x),\ldots,d+\tau(x)$, which give all $d$ types in $\mathbb{Z}/d\mathbb{Z}$. 
\end{proof}

%chocie of polynomial
When $r=1$, we say that $\Gamma \backslash \mathcal{B}$ is \textit{non-partite}. In order to find such Ramanujan complexes we proceed as follows. 
Choose some $\beta \in \mathbb{F}_{q^e}^*$ such that $\beta^d \ne 1$ and that $\alpha = \frac{\beta^d}{1 - \beta^d}$ generates the field $\mathbb{F}_{q^e}$. 
By Lemma 7.2 and the proof of Proposition 7.3 in \cite{LSV2}, if $q^e \geq 4d^2+1$ there exists such $\beta$. 
Now, let $f \in \mathbb{F}_q[y]$ be the minimal polynomial of $\alpha$. Then $f$ is of degree $e$, and under the identification $R/(f) \cong \mathbb{F}_{q^e}$,
$y \longleftrightarrow  \alpha = \frac{\beta^d}{1 - \beta^d}$ and $\frac{y}{1 + y} \longleftrightarrow  \beta^d $.
Therefore $\frac{y}{1 + y} \in (R/I)^{*d}$, and by the discussion above $\Lambda / \Gamma = PSL_d(\mathbb{F}_{q^e})$ and the Cayley complex of $\Lambda / \Gamma$ is non-partite.

The above may be summarized by the following proposition, which is a special case of \cite[Theorem~7.1]{LSV2} together with Lemma \ref{cover}.
\begin{pro} \label{ram}
Let $q$ be a prime power, $d \geq 2$, $e \geq 2$ such that $q^e \geq 4d^2+1$.
Let $\Lambda$ be the Cartwright-Steger lattice in $PGL_d(\mathbb{F}_q((y)))$. 
For an irreducible polynomial $f \in \mathbb{F}_q[y]$, let $\Gamma = \Lambda(f) \lhd \Lambda$ be its congruence subgroup,
and let $\Gamma_0 = \Gamma \cap G_0$ be the finite index subgroup of type-preserving elements in $\Gamma$.

Then there exists an irreducible polynomial $f \in \mathbb{F}_q[y]$ of degree $e$,
such that:\\
1. The Cayley complex $X = Cay(\Lambda / \Gamma, \Sigma)$ is a non-partite Ramanujan complex.\\
2. The Cayley complex $\tilde{X} = Cay(\Lambda / \Gamma_0, \Sigma)$ is a $d$-partite Ramanujan complex.\\
3. The complex $\tilde{X}$ is a $d$-cover of $X$, and the preimage in $\tilde{X}$ of each vertex in $X$ is a set of $d$ vertices of all $d$ types.
\end{pro}

\begin{rem}
A polynomial $f(y)$ of degree $e$ can be also considered as a polynomial of degree $e$ in $\frac{1}{y}$, 
by multiplying it by $\frac{1}{y^e}$, which is an invertible element in $R$.
\end{rem}

%%%%%%%%%%%%%%%%%%%%%%%%%%%%%%%%%%%%%%%%%%%%%%%%%%%%%%%%%%%%%%%%%%%%%%%%%%%%%%%%%%%%%%%%%%%%%%%%%%%%%%%%%%%%%%%%%%%%%%%%%%%%%%%%%%%%%%%%%%%%%%%%%%%%%%%%%%%%%%%%%%%%%%%%%%%%%%%%%%
%%%%%%%%%%%%%%%%%%%%%%%%%%%%%%%%%%%%%%%%%%%%%%%%%%%%%%%%%%%%%%%%%%%%%%%%%%%%%%%%%%%%%%%%%%%%%%%%%%%%%%%%%%%%%%%%%%%%%%%%%%%%%%%%%%%%%%%%%%%%%%%%%%%%%%%%%%%%%%%%%%%%%%%%%%%%%%%%%%

\section{Proof of the main theorem}\label{sec:Proof of the main theorem}
We are now ready to prove the following theorem which implies Theorem \ref{chromatic-bound}.

\begin{thm} \label{main}
Let $d \geq 3 $ and $q$ an odd prime power, and let $X$ be a non-partite Ramanujan complex as constructed in Proposition \ref{ram}.
Let $\chi(X)$ and $r(X)$ be the chromatic number and injectivity radius of $X$ (as defined in the introduction).
Then 
\begin{equation}
r(X) \geq \frac{\log_q|X|}{2(d-1)(d^2-1)} - \frac{1}{2}
\end{equation}
and, assuming $r(X) \geq 2$, 
\begin{equation}
\chi(X) \geq \frac{1}{2} \cdot q^{1/2d}
\end{equation}
\end{thm}

\begin{proof}
The claim about the injectivity radius is Corollary \ref{inj-rad}. For the chromatic number:
Consider a coloring of $X$ with $\chi(X)$ colors, and let $W$ be the set of vertices of $X$ of the most common color, so $|W| \geq \frac{|V(X)|}{\chi(X)}$.
Let $\widetilde{X}$ be the $d$-cover of $X$ with the type-function inherited from building, as in Proposition \ref{ram}. 
For each $i \in \mathbb{Z}/d\mathbb{Z}$, let $W_i$ be the preimage of $W$ of vertices of type $i$ in $\widetilde{X}$, so $|W_i|=|W|$.

Note that the image of each $j$-dimensional simplex $\tilde{e}$ in $\widetilde{X}$ is again an $j$-dimensional simplex $e$ in $X$. 
Indeed, as the injectivity radius is $\geq 2$ and each simplex is a clique, so any two vertices in $\tilde{e}$ are of distance $1$, 
hence are not mapped to the same vertex in $X$.

In particular, each $d$-dimensional simplex in $\widetilde{X}$, with one vertex in each $W_i$, is mapped to a $d$-dimensional simplex in $X$, with all the vertices in $W$.
But by the definition of the chromatic number there are no such simplices in $X$ with all vertices in $W$, and therefore  $E(W_1,\ldots,W_d) = \emptyset$.

Denote by $V_i$ the set of vertices of type $i$ in $\widetilde{X}$, then $|V_i| = |V(X)|$ for all $i=1,\dots,d$. Therefore $\frac{|W_i|}{|V_i|} = \frac{|W|}{|V(X)|} \geq \frac{1}{\chi(X)}$.  Since $E(W_1,\ldots,W_d) = \emptyset$, we get 
$$disc_{\widetilde{X}}(W_1,\ldots,W_d) = \prod_{i=1}^d \frac{|W_i|}{|V_i|} \geq \frac{1}{\chi(X)^d}.$$
On the other hand, by the Colorful Mixing Lemma, we have 
$$disc_{\widetilde{X}}(W_1,\ldots,W_d) \leq \frac{2d}{q^{1/2}}.$$
Combining these together we get 
$$\chi(X) \geq  (2d)^{-\frac{1}{d}} \cdot q^{\frac{1}{2d}} \geq \frac{1}{2} \cdot q^{\frac{1}{2d}}.$$
\end{proof}

\begin{rem} \label{diameter}
The complexes in Theorem \ref{main} are non-partite. It follows therefore that for their $1$-skeletons, the largest eigenvalue of their adjacency matrices  is 
$k \approx q^{d^2/4}$ where the second one $\lambda_2$ is at most $d^d q^{d^2/8}$ (see \cite[remark~2.1.5]{Lu2}).
It follows from \cite[Theorem~1]{C} that $diam(X) \leq \frac{\log|X|}{\log(\lambda_1/\lambda_2)}$, so
$$diam(X_n) \leq \frac{\log_q |X_n|}{\log_q(\frac{k}{d^d q^{d^2/8}})} \approx \frac{\log_q |X_n|}{d^2/4 - d\log_q d} \approx \frac{4 \log_q |X_n|}{d^2} \leq 8d \cdot r(X_n)$$ 
for $q \gg d$. So, up to a constant fraction of their diameters, these complexes are two colorable around every vertex.
\end{rem}

\begin{rem}
Theorems \ref{main} and \ref{chromatic-bound} are true also if $q$ is even or if $d=2$. 
In this case one should use the full power of the Ramanujan bounds. For the cases we treated here, Oh's theorem suffices.
\end{rem}

%%%%%%%%%%%%%%%%%%%%%%%%%%%%%%%%%%%%%%%%%%%%%%%%%%%%%%%%%%%%%%%%%%%%%%%%%%%%%%%%%%%%%%%%%%%%%%%%%%%%%%%%%%%%%%%%%%%%%%%%%%%%%%%%%%%%%%%%%%%%%%%%%%%%%%%%%%%%%%%%%%%%%%%%%%%%%%%%%%
%%%%%%%%%%%%%%%%%%%%%%%%%%%%%%%%%%%%%%%%%%%%%%%%%%%%%%%%%%%%%%%%%%%%%%%%%%%%%%%%%%%%%%%%%%%%%%%%%%%%%%%%%%%%%%%%%%%%%%%%%%%%%%%%%%%%%%%%%%%%%%%%%%%%%%%%%%%%%%%%%%%%%%%%%%%%%%%%%%

%%%%%%%%%%%%%%%%%%%%%%%%%%%%%%%%%%%%%%%%%%%%%%%%%%%%%%%%%%%%%%%%%%%%%%%%%%%%%%%%%%%%%%%%%%%%%%%%%%%%%%%%%%%%%%%%%%%%%%%%%%%%%%%%%%%%%%%%%%%%%%%%%%%%%%%%%%%%%%%%%%%%%%%%%%%%%%%%%%
\end{document}